\documentclass[twoside,11pt,reqno,a4paper]{amsart}

\usepackage{amsthm, amsmath, amssymb, graphicx,enumerate}

\usepackage[usenames,dvipsnames]{color}
\usepackage{algorithmic, comment}
\usepackage{listings}
\usepackage[all]{xy}

\numberwithin{equation}{section}

\newcommand{\bfw}{{\mathbf{w}}} 
\newcommand{\ffc}{{\boldsymbol{\mathfrak{c}}}} 
 \newcommand{\ft}{{\boldsymbol{\mathfrak{t}}}}
\newcommand{\bfc}{{\boldsymbol{c}}} \newcommand{\bft}{{\boldsymbol{t}}}
\DeclareMathOperator{\Ext}{Ext}

\renewcommand{\phi}{\varphi}

\newcommand{\ZZ}{\mathbb{Z}}

\newtheorem{theorem}{Theorem}
\newtheorem*{thm*}{Theorem}
\newtheorem{lemma}[theorem]{Lemma}
\newtheorem{cor}[theorem]{Corollary}
\newtheorem{definition}[theorem]{Definition}
\newtheorem*{notation}{Notation}
\newtheorem{remark}[theorem]{Remark}
\newtheorem{proposition}[theorem]{Proposition}
\newtheorem{example}[theorem]{Example}

\renewcommand{\leq}{\leqslant}
\renewcommand{\le}{\leqslant}
\renewcommand{\geq}{\geqslant}
\renewcommand{\ge}{\geqslant}

\theoremstyle{definition}

\newcommand{\fc}{\ensuremath{\mathfrak c}}

\newcommand{\dif}{\mathrm{d}}

\title[Ext-groups of Weyl modules for $GL_2$]{Dimensions of Ext-groups of Weyl modules for $GL_2$}

\author[Stephan Baier]{\sc Stephan Baier}
\address{School of Mathematics\\ University of East Anglia\\
Norwich\\ NR4 7TJ, UK}
\email{s.baier@uea.ac.uk}

\author[Sergey Lamzin]{\sc Sergey Lamzin}
\address{The Genome Analysis Centre\\ Norwich Research Park\\ 
Norwich\\ NR4 7UH, UK}
\email{s.lamzin@uea.ac.uk}

\author[Vanessa Miemietz]{\sc Vanessa Miemietz}
\address{School of Mathematics\\ University of East Anglia\\
Norwich\\ NR4 7TJ, UK}
\email{v.miemietz@uea.ac.uk}

\thanks{The third author acknowledges support from ERC grant PERG07-GA-2010-268109.}

\begin{document}
\maketitle

\section{Introduction}

Let $\mathbb{F}$ be an algebraically closed field of positive characteristic $p$. The homological algebra associated to the rational representation theory of algebraic groups over $\mathbb{F}$ has been an object of study for many years. In particular, extension groups of Weyl modules and simple modules have been investigated intensively \cite{CE, E, EHP, MT3, MT, AP, PS, S}.
In \cite{MT}, the third author and Will Turner gave an explicit description of the extension algebra of Weyl modules for $GL_2(\mathbb{F})$ via an iterative application of certain $2$-functors to the smallest object of a $2$-category. This, in particular, produced an explicit basis, which is multiplicative up to specified sign. While the focus of said paper was to understand the algebra structure of this extension algebra, an algorithm for determining the dimension of a Ext-group having previously been given in \cite[Theorem 5.1]{AP}, the present paper examines the basis given in \cite{MT} and uses it to give upper and lower bounds for the growth behaviour of the dimension of Ext-groups. In particular, we answer positively the question raised in \cite[Section 1.2]{EHP} whether dimensions of groups of $k$-extensions (for fixed $k$) are bounded independently of the highest weights of the two standard modules. We furthermore implement an algorithm determining the dimensions of any given Ext-group. 

The fundamental difference between our algorithm and the recursion given in \cite{AP} is that, while each step in the recursion given in \cite{AP} produces summand in various blocks in every step and hence many zero contributions, our algorithm consists of two steps: the first determines whether two highest weights $\mu$ and $\lambda$ of Weyl modules are in the same block, and if they are, assigns to them their numbers $m$ and $\ell$ in the total order on weights in this block; the second step then focuses on an abstract incarnation of a block with sufficiently many simples, and computes extensions between the $m$th and the $\ell$th standard modules, denoted by $\Delta_m, \Delta_\ell$ respectively, therein. Similarly, the growth rate is investigated in terms of $m$ and $\ell$ inside a given block. 
As an application of our algorithm, we obtain a duality formula (Theorem \ref{duality}), which we do not know any reference for and believe to be new.

We proceed as follows:
In Section \ref{alg}, we introduce some of the algebraic notation and outline some of the ideas from \cite{MT}. We explain how the basis of the whole Ext-algebra that we use to compute dimensions of Ext-groups arises as a subspace of a product of certain polytopes, which components it has, as well as their algebraic meaning, without, however giving detailed descriptions.  These are then presented Section \ref{polytopes}, where we give full combinatorial details of the basis of the Ext-algebra given in \cite{MT}, consisting of tuples of elements coming from various polytopes in $\ZZ^7$.  Section \ref{secduality} then gives the duality formula mentioned above.
Section \ref{analysis} analyses the polytopes more closely and streamlines their presentation.
Section \ref{kbasis} investigates which tuples of such polytopal elements can actually appear in the basis of the Ext-algebra, i.e. which lie in the subspace mentioned above. Section \ref{compdim} reduces the tasks of computing the dimension of $\Ext^k(\Delta_m, \Delta_\ell)$ to the computation of certain quantities, for which recursions are developed in Section \ref{secrecursions}. Sections \ref{redtopartitions} gives an explicit formula for these quantities in terms of a partition function (cf.\ Theorem \ref{Bexplicit}), whose growth behaviour is then investigated in Section \ref{partinvest}. The results from the latter are then used in Section \ref{secestimates} to give lower (Theorem \ref{lowerthm}) and upper (Theorem \ref{upperthm}) bounds for $ \dim \Ext^k(\Delta_m,\Delta_\ell)$ for  any fixed $k$ independently of $m$ and $l$. Section \ref{secalg} outlines the algorithm used to compute the actual dimension, which is then implemented into a C program in Section \ref{program}. Finally we add the short program 
determining $m$ and $\ell$ from highest weights $\mu$ and $\lambda$ in Section \ref{lambdamu}.

\section{Algebraic setup}\label{alg}

Let us first explain some of the algebraic background feeding into the basis for the Yoneda extension algebra of Weyl modules for $GL_2(\mathbb{F})$ given in \cite{MT}. 
It was first observed in \cite{MT2}, that blocks of rational representations of $GL_2(\mathbb{F})$ can be obtained via iterated application of certain algebraic operators $\mathbb{O}_{\bfc,\bft}$ which are indeed $2$-endofunctors on a certain $2$-category. Here $\bfc$ is the algebra describing a block of polynomial representations of $GL_2(\mathbb{F})$ with $p$ simple modules and $\bft$ is its characteristic tilting module, which, since $\bfc$ is Ringel self-dual, is in fact a $\bfc$-$\bfc$-bimodule. This algebra $\bfc$ is also radical graded and the 
operator $\mathbb{O}_{\bfc,\bft}$ takes a pair consisting of an algebra $A$ and an $A$-$A$-bimodule $M$ to the pair
$$
\mathbb{O}_{\bfc,\bft}(A,M)=(\bigoplus_{i \in \ZZ_{\geq 0}} \bfc^j \otimes_\mathbb{F} M^{\otimes_A j} ,\bigoplus_{i \in \ZZ_{\geq 0}} \bft^j \otimes_\mathbb{F} M^{\otimes_A j} )
$$
where $\bfc^j, \bft^j$ denote the $j$th graded piece of $\bfc, \bft$ respectively, and $M^{\otimes_A 0}$ is interpreted as $A$.
A block of polynomial representations with $p^q$ simple modules is equivalent to $\mathbb{O}_{\mathbb{F},0}\mathbb{O}_{\bfc,\bft}^q(\mathbb{F},\mathbb{F})$-$\mathrm{mod}$ (where $\mathbb{O}_{\mathbb{F},0}$ simply takes the algebra part of a pair $(A,M)$), and a block of rational representations is obtained as the inverse limit for $q \to \infty$ (since $\mathbb{O}_{\mathbb{F},0}\mathbb{O}_{\bfc,\bft}^q(\mathbb{F},\mathbb{F})\twoheadrightarrow\mathbb{O}_{\mathbb{F},0}\mathbb{O}_{\bfc,\bft}^{q-1}(\mathbb{F},\mathbb{F})$ by the general theory of quasihereditary algebras). 

In the present article, we restrict our attention to blocks of  polynomial representations. Given two Weyl modules, there will be a number $q$, such that both are contained in a block of polynomial representations with $p^q$ simple modules. The extension algebra of Weyl modules for such a block of polynomial representations is a subalgebra of the extension algebra of  Weyl modules in the category of all rational representations, and we can hence concentrate on taking the extension algebra of Weyl (=standard) modules of $\mathbb{O}_{\mathbb{F},0}\mathbb{O}_{\bfc,\bft}^q(\mathbb{F},\mathbb{F})$ for appropriate $q$.

The algebra $\bfc$ is a generalised Koszul algebra (\cite{Mad}), thus the Yoneda extension algebra of its Weyl modules modules is easily computed, an algebra we will denote  by $\ffc$ in the following. While for $q>1$, $\mathbb{O}_{\mathbb{F},0}\mathbb{O}_{\bfc,\bft}^q(\mathbb{F},\mathbb{F})$ is not Koszul, it was proved in \cite[Theorem 18, Proposition 21]{MT}, that its Yoneda extension algebra of Weyl modules can be obtained as $\mathbb{H}\mathbb{O}_{\mathbb{F},0}\mathbb{P}_{\ffc,\underline{\ft}}^q(\mathbb{F},\mathbb{F})$, where $\mathbb{P}$ is a generalisation of the operator $\mathbb{O}$ that takes positive and negative degrees into account, and $\mathbb{H}$ simply takes homology. The subscript $\underline{\ft} = (\ft, \ft^{-1})$ in the subscript of $\mathbb{P}$ refers to  a pair of differential graded bimodules obtained by pushing  (a projective bimodule resolution of) $\bft$ through generalised Koszul duality to give $\ft$ and taking its adjoint $\ft^{-1}$.

It is further observed in \cite[Subsections 5.6, 5.7]{MT} that $\mathbb{H}\mathbb{O}_{\mathbb{F},0}\mathbb{P}_{\ffc,\underline{\ft}}^q(\mathbb{F},\mathbb{F})$ is isomorphic to a certain subalgebra $\bfw_q$ in the $q$-fold tensor product of the tensor algebra $\Upsilon = \mathbb{HT}_{\ffc}(\underline{\ft})$ with itself. More precisely, the algebra $\Upsilon$ is triply graded (we denote these degrees by $i,j,k$ respectively), where the $i$-grading is just the tensor grading placing $\ft$ in degree $1$ and $\ft^{-1}$ in degree $-1$, the $j$-grading is an algebraic grading on $\ffc, \underline{\ft}$ obtained by pushing the grading on $\bfc$ through generalised Koszul duality, and the $k$-grading describes a homological grading, in particular the differential on $\bft$ has $k$-degree $1$.

The subalgebra $\bfw_q$ of $\Upsilon^{\otimes q}$ is then that generated by all homogeneous elements $v_1\otimes \cdots \otimes v_q$ where the $i$-degree of $v_1$ is zero (i.e.\ $v_1$ comes from the algebra component of $\Upsilon$), and for all subsequent elements $v_g$ their $i$-degree $i_g$ equals the $j$-degree $j_{g-1}$ of the preceding element $v_{g-1}$.

In order to obtain an explicit description of $\bfw_q$, the authors of \cite{MT} therefore explicity compute large parts of the algebra $\Upsilon$, observing that only the subspace $\Upsilon^{\leq 1} = \mathbb{H}(\ft\mathbb{T}_{\ffc}(\ft^{-1}))$ features in the subalgebra of $\Upsilon^{\otimes q}$ we are interested in \cite[Lemma 29]{MT}.
Viewed as a $\ffc$-$\ffc$-bimodule, the subspace $\Upsilon^{\leq 1}$ has direct summands 
\begin{itemize}
\item the algebra $\ffc$, as well as its regular bimodule;
\item a $\ffc$-$\ffc$-bimodule $M$, which is an extension of $\ffc^*$ (the dual bimodule) by $\ffc$;
\item a subbimodule $\overline{M}$ thereof;
\item the $\ffc$-$\ffc$-bimodule $\overline{\ffc^{0\sigma}}$ which is a truncation of the semisimple quotient of $\ffc$ with the right action twisted by an automorphism $\sigma$;
\item the full semisimple quotient $\ffc^{0\sigma}$ (with the right action again twisted by $\sigma$) of $\ffc$.
\end{itemize}

A more detailed picture of the $\ffc$-$\ffc$-bimodule decomposition of $\Upsilon^{\leq 0}:= \mathbb{H}\mathbb{T}_{\ffc}(\ft^{-1})$ is given by the picture
\begin{equation}\label{picture}
\xymatrix@C=1pt@R=2pt{
    &        &        &             & \fc &          \\
    &        &        & \overline{M}^\tau & \oplus    &  \overline{\ffc^{0\sigma}}\\
    &        &  M   &  \oplus    & \ffc &      \\
    &   M^\tau& \oplus & \overline{M}^\tau  & \oplus  &  \overline{\ffc^{0\sigma}}  \\
M   & \oplus &   M& \oplus      & \ffc &         \\
    &        &        &             &......&          \\
}
\end{equation}
where the row indicates the tensor degree (the first row  just showing $\mathbb{H}(\ffc)=\ffc$, the second row showing $\mathbb{H}(\ft^{-1})$, the third row showing $\mathbb{H}((\ft^{-1})^{\otimes_{\ffc}2})$, etc.). 

This picture holds in case $p>2$, however, the only thing that changes in case $p=2$ is that there will be an indecomposable extension between $\overline{M}$ and $\overline{\ffc^{0\sigma}}$, changing the bimodule structure, but not the given combinatorial description of a basis, making our algorithm still valid in this case.

All bimodules appearing in this decomposition have bases indexed by elements $(s, j_0,k_0,t)$ of certain polytopes in $\ZZ^4$ (see \cite[Lemma 51]{MT}, here Definition \ref{smallpolys}), where $s$ and $t$ denote the numbers of the idempotents such that $e_s(s, j_0,k_0,t)e_t \neq 0$ and $j_0, k_0$ denote the $j$- resp. $k$-degrees in the gradings described above.
Taking into account the position of the respective bimodule in \eqref{picture}, we arrive at a basis $\mathcal{P}_{\Upsilon^{\le 0}}$ for $\Upsilon^{\le 0}$ of certain $7$-tuples of integers $(s,i,j,k,a,b,t)$ (see Definition \ref{def3}), where $s,t$ are as above, $a$ counts with northwest to southeast diagonal the bimodule occurs in in \eqref{picture} (starting with $a=0$ on the top right), $b$ counts which northeast to southwest diagonal the bimodule occurs in in \eqref{picture} (starting with $b=0$ on the top left), $i$ denotes the tensor degree (hence necessarily $i=-a-b$), and $j$ and $k$ are obtained from $j_0,k_0$ by certain formulae (see \eqref{ijkconvert}) which come from grading shifts on the bimodules computed in \cite[Subsection 9.6]{MT}. To this we need to add a basis of $\mathbb{H}\ft \cong  \ffc^{0 \sigma}$ to obtain a basis $\mathcal{P}_{\Upsilon^{\le 1}}$ for $\Upsilon^{\le 1}$. Full details of these bases will be given in the next section.
The subalgebra $\bfw_q$ of $\Upsilon^{\otimes q}$ which is isomorphic to the extension algebra of Weyl modules for a block of polynomial representations of $GL_2(\mathbb{F})$ with $p^q$ simple modules then has a basis given by all $q$-tuples  of basis elements $(s_g,i_g,j_g,k_g,a_g,b_g,t_g)$ for $\Upsilon^{\leq 1}$ ($1 \leq g \leq q$), where $i_1 = 0$ and $i_g = j_{g-1}$ for $g=2,\dots q$. Such an element will determine a $k$-extension for  $k_1 + \cdots + k_q = k$ between the $m$th and the $\ell$th modules in the block where 
\begin{equation*}
\begin{split}
m= (s_1-1)p^{q-1}+(s_2-1)p^{q-2}+...+(s_{q-1}-1)+s_q \\
\ell=(t_1-1)p^{q-1}+(t_2-1)p^{q-2}+...+(t_{q-1}-1)p+t_q.
\end{split}
\end{equation*}

\section{Polytopal basis for $\Upsilon^{\le 1}$ and $\Ext^k(\Delta_m, \Delta_\ell)$}\label{polytopes}
In the following, we recall the explicit description of a polytopal basis for $\Upsilon^{\le 1}$, as given in \cite[Subsection 9.8]{MT}. First, we define the polytopes on which this description is based. 

\begin{definition}\label{smallpolys} Define
\begin{equation*}
\begin{split}
\mathcal{P}_{\ffc} &=  \left\{ (s,j_0,k_0,t)\in \mathbb{Z}^4\  :\  \begin{array}{l} 1\le s\le t\le p, \ 0\le j_0+k_0\le 1, \\  t-s=j_0+2k_0,\ j_0=0=k_0\mbox{\rm\ if } s=t \end{array}
\right\},\\
\mathcal{P}_0& =  \left\{(s,j_0,k_0,t)\in \mathbb{Z}^4\ :\    \begin{array}{l}  1\le s,t\le p, \ s+t=p+1, \\j_0=k_0=0 \end{array}\right\}\setminus \{(p,0,0,1)\},\\
\mathcal{P}_M &= \left\{(s,j_0,k_0,t)\in \mathbb{Z}^4\ :\  \begin{array}{l} 1\le s,t\le p, \ j_0+2k_0+2=t-1-s+p,\\  0\le j_0+k_0+2\le 1\end{array}\right\},\\
\mathcal{P}_{\overline{M}} &=  \mathcal{P}_M\setminus\{(p,0,-1,1)\}.
\end{split}
\end{equation*}
\end{definition} 

\begin{example}\label{ex1}
The following is a diagram of the polytope $\mathcal{P}_M$ in case $p=3$ (we depict its structure as a left module):
{\footnotesize$$
\xymatrix@R=3pt@C=1pt{
31^{-1}_0 \ar@{-}[dr]&& 31_{-2}^0 \ar@{-}[dl]\ar@{-}[dr] && \\
&21_{-1}^0 \ar@{-}[dr]&& 21_{-3}^1 \ar@{-}[dl]\ar@{-}[dr]& \\
&&11_{-2}^1 && 11_{-4}^2 
}\hspace*{-2mm}\xymatrix@R=3pt@C=1pt{
        32_{-1}^0 \ar@{-}[dr]&&32_{-3}^1 \ar@{-}[dl]\ar@{-}[dr]&& \\
&22_{-2}^1 \ar@{-}[dr]&& 22_{-4}^2 \ar@{-}[dl]\ar@{-}[dr]&\\
 &&12_{-3}^2 &&12_{-5}^3 
}\hspace*{-2mm}
\xymatrix@R=3pt@C=1pt{
33_{-2}^1 \ar@{-}[dr]&&33_{-4}^2 \ar@{-}[dl]\ar@{-}[dr]& &\\
&23_{-3}^2 \ar@{-}[dr]&&23_{-5}^3 \ar@{-}[dl]\ar@{-}[dr]&\\
&&13_{-4}^3 &&13_{-6}^4.
}$$}
In the diagram an element $(s,j_0,k_0,t)$ is written $st^{k_0}_{j_0}$.
\end{example}

Further, we define a set of vectors in $\mathbb{Z}^6$ related to these polytopes.

\begin{definition} Define
\begin{equation*}
\mathcal{M}:=\mathcal{M}_1\cup \mathcal{M}_2\cup \mathcal{M}_3\cup \mathcal{M}_4,
\end{equation*}
where 
\begin{equation*}
\begin{split}
\mathcal{M}_1:= & \left\{(s,j_0,k_0,a,b,t)\in \mathbb{Z}^6\ :\ (s,j_0,k_0,t)\in \mathcal{P}_{\ffc},\ a,b\ge 0,\ a=b\right\},\\
\mathcal{M}_2:= & \left\{(s,j_0,k_0,a,b,t)\in \mathbb{Z}^6\ :\ (s,j_0,k_0,t)\in \mathcal{P}_0,\ a,b\ge 0,\ a=b-1\right\},\\
\mathcal{M}_3:= & \left\{(s,j_0,k_0,a,b,t)\in \mathbb{Z}^6\ :\ (s,j_0,k_0,t)\in \mathcal{P}_{\overline{M}}, \ a,b\ge 0, \ a=b+1\right\},\\
\mathcal{M}_4:= & \left\{(s,j_0,k_0,a,b,t)\in \mathbb{Z}^6\ :\ (s,j_0,k_0,t)\in \mathcal{P}_{M}, \ a,b\ge 0,\ a>b+1\right\}.
\end{split}
\end{equation*}
\end{definition} 

Using the set $\mathcal{M}$ defined above, we define the following sets of vectors in $\mathbb{Z}^7$.

\begin{definition}\label{def3} For given $(j_0,k_0,a,b)\in \mathbb{Z}^4$ with $a\ge b-1$ set
\begin{equation} \label{ijkconvert}
\begin{split}
i:=& -a-b,\\
j:=& 
\begin{cases}
j_0-(a-b-1)p+1 & \mbox{\rm\ if } a\ge b+1, \\
j_0 & \mbox{\rm\ if } a=b,\\
j_0+1 & \mbox{\rm\ if } a=b-1,
\end{cases}\\
k:=&
\begin{cases}
k_0+(a-b-1)(p-1) & \mbox{\rm\ if } a\ge b+1, \\
k_0 & \mbox{\rm\ if } a\le b.
\end{cases}
\end{split}
\end{equation}
Then define 
\begin{equation*}
\mathcal{P}_{\Upsilon^{\le 0}}:= \left\{(s,i,j,k,a,b,t)\in \mathbb{Z}^7\ :\ (s,j_0,k_0,a,b,t)\in \mathcal{M}\right\}
\end{equation*}
and
\begin{equation*}
\mathcal{P}_{\Upsilon^{\le 1}}:= \mathcal{P}_{\Upsilon^{\le 0}}\cup \{(s,1,1,0,0,0,p+1-s)\}.
\end{equation*}
\end{definition} 

Now \cite[Theorem 53]{MT} gives us the following polytopal basis for $\Upsilon^{\le 1}$.

\begin{theorem} $\Upsilon^{\le 1}$ has a basis of the form $\left\{ m_w\right\}_{w\in \mathcal{P}_{\Upsilon^{\le 1}}}$.
\end{theorem}

For any positive integer $q$, we are interested in the following subspace of the $q$-fold tensor product of $\Upsilon^{\le 1}$ with itself, defined below. 

\begin{definition}\label{longdef} Let $1\le m\le p^q$, $1\le \ell\le p^q$ and $k\in \mathbb{Z}$ be given. Let $s_1,...,s_q,t_1,...,t_q\in \{1,...,p\}$ be the uniquely determined numbers such that
\begin{eqnarray}
m & = & (s_1-1)p^{q-1}+(s_2-1)p^{q-2}+...+(s_{q-1}-1)p+s_q, \label{m} \\
\ell  & = & (t_1-1)p^{q-1}+(t_2-1)p^{q-2}+...+(t_{q-1}-1)p+t_q. \label{e}
\end{eqnarray}
Set $j_0:=0$ and define 
\begin{equation*}
\begin{split}
\mathcal{B}^k(m,\ell):=\left\{({\bf v}_1,{\bf v}_2,...,{\bf v}_q) :   \begin{array}{l}\mbox{\bf v}_g=(s_g,i_g,j_g,k_g,a_g,b_g,t_g)  \in \mathcal{P}_{\Upsilon^{\le 1}}\\ \text{\rm with } i_g=j_{g-1}  \ \text{\rm  for } g \in \{1,...,q\},\\ k_1+\cdots +k_q=k \end{array}\right\}.
\end{split}
\end{equation*}
\end{definition}

Then the subset $\mathcal{B}^k(m,\ell)$ forms a basis of the subspace $\Ext^k(\Delta_m,\Delta_\ell)$ in $\bfw_q$ by \cite[Section 2]{MT}.

\section{A duality formula}\label{secduality}

Throughout the following, let $q\ge 2$ be a fixed natural number.
Before going on to compute the cardinality of $\mathcal{B}^k(m,\ell)$, we establish an interesting duality property which the authors became aware of by looking at the results of their computer calculations.

\begin{theorem}\label{duality} For all $m,\ell\in \{1,...,p^q\}$ and $k\in \{0,...,p^q-1\}$, we have
$$
\dim \Ext^k(\Delta_m,\Delta_\ell)=\dim\Ext^k\left(\Delta_{\tilde{m}},\Delta_{\tilde{\ell}}\right),
$$
where
$$
\tilde{\ell}= p^q+1-m \quad \mbox{and} \quad \tilde{m}=p^q+1-\ell.
$$
\end{theorem}

\begin{proof} $ $\medskip\\
{\it Combinatorial proof:}
We need to show that there is a bijection between $\mathcal{B}^k(m,\ell)$ and $\mathcal{B}^k(\tilde{m},\tilde{\ell})$. Define $s_1,...,s_q,t_1,...,t_q$ as in Definition \ref{longdef} and, similarly,
$\tilde{s}_1,...,\tilde{s}_q,\tilde{t}_1,...,\tilde{t}_q$ corresponding to $\tilde{m}$ and $\tilde{\ell}$. We observe that then
$$
\tilde{s}_i:=p+1-t_i \mbox{ and } \tilde{t}_i:=p+1-s_i
$$
for all $i\in \{1,...,q\}$. Now it suffices to show that 
\begin{equation} \label{equi}
(s,i,j,k,a,b,t)\in \mathcal{P}_{\Upsilon^{\le 1}} \Longleftrightarrow (\tilde{s},i,j,k,a,b,\tilde{t})\in \mathcal{P}_{\Upsilon^{\le 1}},
\end{equation}
where
$$
\tilde{s}:=p+1-t \quad \mbox{and} \quad \tilde{t}:=p+1-s.
$$
Indeed, noting that $\tilde{t}-\tilde{s}=t-s$ and
$$
t=p+1-s \Longleftrightarrow \tilde{s}=s \mbox{ and } \tilde{t}=t,
$$
we see that
$$
(s,i,j,k,a,b,t)\in \mathcal{S}_i \Longleftrightarrow (\tilde{s},i,j,k,a,b,\tilde{t})\in \mathcal{S}_i \quad \mbox{for } i\in \{1,2,3\}
$$
which implies \eqref{equi} by Proposition \ref{pupsilon}, completing the proof. \medskip\\
{\it Sketch of algebraic proof:}  
By \cite[Remark 23]{MT}, the extension algebra of Weyl modules for a block with $p$ simple modules is Koszul selfdual.
In this small case, we are interested in $(e_i\ffc e_j)^k$, which under the algebra isomorphism  $\ffc\to \ffc^!$ given in \cite[Remark 23]{MT} maps to $(e_{p+1-i}\ffc^! e_{p+1-j})^k$. As the quiver for $\ffc^!$ is just the opposite of the one for $\ffc$, it is apparent that the dimension of $(e_{p+1-i}\ffc^!e_{p+1-j})^k$ is the same as the dimension of $(e_{p+1-j}\ffc e_{p+1-i})^k$ (by left-right duality), which in this case is what we need. Pushing the bimodules appearing in the construction through this algebra isomorphism followed by left-right duality, gives the desired result in general.
\end{proof}

\section{Analysis of basis vectors of $\Upsilon^{\le 1}$}\label{analysis}

In the following, we analyse the $7$-tuples appearing in $\mathcal{P}_{\Upsilon^{\le 1}}$ more explicitly, which gives rise to a more uniform combinatorial presentation. To this end, we need the following sets.

\begin{definition} \label{Sdef} We define sets $\mathcal{S}_1, \mathcal{S}_2, \mathcal{S}_3$ by
\begin{equation*}
\begin{split}
\mathcal{S}_1 &:=  \left\{ (s,i,j,k,a,b,t)\in \mathbb{Z}^7 \ :\begin{array}{l} 1\le s\le p, \ a\ge b\ge 0, \ 1\le t\le p, \\ t-s\ge 0 \ \text{\rm if } a-b=0, \ i=-a-b, \\
j=-p(a-b)-(t-s)+2u\ \text{\rm and } \\  k=(p-1)(a-b)+(t-s)-u \\ \text{\rm with } u\in\{0,1\}, \\
 u=0 \ \text{\rm if }t-s=0 \ \text{\rm and }a-b=0, \\  
 t-s\ge 2-p\ \text{\rm if }u=1\ \text{\rm and }a-b=1 \end{array} \right\},\\
\mathcal{S}_2 &:=  \left\{ (s,i,j,k,a,b,t) \in \mathbb{Z}^7 \ :\begin{array}{l} 1\le s\le p-1,\ t=p+1-s,\ a\ge 0, \\ b=a+1\ i=-2a-1,\ j=1,\ k=0 \end{array}\right\}\\
\mathcal{S}_3& :=  \left\{ (s,i,j,k,a,b,t)\in \mathbb{Z}^7\ : \begin{array}{l} 1\le s\le p,\ i=1,\ j=1,\ k=0,\\ a=0,\ b=0,\ t=p+1-s\end{array}\right\}.
\end{split}
\end{equation*}
\end{definition}

We shall prove the following.

\begin{proposition} \label{pupsilon} We have
$$
\mathcal{P}_{\Upsilon^{\le 1}}=\mathcal{S}_1\cup \mathcal{S}_2\cup \mathcal{S}_3,
$$
and the sets $\mathcal{S}_1$, $\mathcal{S}_2$ and $\mathcal{S}_3$ are disjoint.
\end{proposition}

\begin{proof} We divide our analysis into several cases. 
\begin{itemize}
\item
{\bf Case i:} First, we investigate vectors $(s,i,j,k,a,b,t)\in \mathcal{P}_{\Upsilon^{\le 1}}$ which come from vectors $(s,j_0,k_0,a,b,t)\in \mathcal{M}_1$, i.e. we consider vectors $(s,j_0,k_0,a,b,t)\in\mathbb{Z}^6$
with $1\le s\le t\le p$, $0\le j_0+k_0\le 1$, $a=b\ge 0$, $t-s=j_0+2k_0$ and $j_0=k_0=0$ if $s=t$. Since $t-s-2k_0=j_0$, we conclude that 
$$
0\le t-s-k_0\le 1
$$ 
and hence 
$$
k_0=t-s-u,\ j_0=-(t-s)+2u  
$$
with $u\in\{0,1\}$, and $u=0$ if $s=t$. Therefore, these vectors are of the form
$$
(s,j_0,k_0,a,b,t)=(s,-(t-s)+2u,t-s-u,a,a,t).
$$
Under the operations in \eqref{ijkconvert}, said vectors transform into elements of $\mathcal{P}_{\Upsilon^{\le 1}}$ of the form
$$
(s,i,j,k,a,b,t)=(s,-2a,-(t-s)+2u,t-s-u,a,a,t),
$$
where $u\in\{0,1\}$, $1\le s\le t\le p$, $a\ge 0$ and $u=0$ if $s=t$.
\item
{\bf Case ii:} Second, we investigate vectors $(s,i,j,k,a,b,t)\in \mathcal{P}_{\Upsilon^{\le 1}}$ that come from vectors $(s,j_0,k_0,a,b,t)\in \mathcal{M}_3\cup \mathcal{M}_4$, i.e. 
vectors $(s,j_0,k_0,a,b,t)\in \mathbb{Z}^6$ 
with $1\le s,t\le p$, $0\le j_0+k_0+2\le 1$, $t-s-1+p=j_0+2k_0+2$, $a,b\ge 0$, $a\ge b+1$, where we exclude the vectors $(p,0,-1,a,a-1,1)$ with $a\ge 1$. Since $t-s-2k_0-3+p=j_0$, we conclude that
$$
0\le t-s-k_0+p-1\le 1
$$
and hence
$$
k_0=t-s+p-1-u,\ j_0=-(t-s)-p-1+2u 
$$
with $u\in\{0,1\}$.
Therefore, these vectors are of the form
\begin{equation*}\begin{split}
(s,j_0,k_0,a,b,t)&=(s,-(t-s)-p-1+2u,t-s+p-1-u,a,a-1,t)\\ &\not=(p,0,-1,a,a-1,1).
\end{split}\end{equation*}
Under the operations in \eqref{ijkconvert}, said vectors transform into elements of $\mathcal{P}_{\Upsilon^{\le 1}}$ of the form
\begin{equation*}
(s,i,j,k,a,b,t)= (s,-a-b,-p(a-b)-(t-s)+2u,(p-1)(a-b)+(t-s)-u,a,b,t),
\end{equation*}
where $u\in\{0,1\}$, $1\le s,t\le p$, $a\ge b+1\ge 1$, and $t-s\ge 2-p$ if $u=1$ and $b=a-1$. The last condition ensures that we allow only vectors from $\mathcal{P}_{\overline{M}}$ if the vectors come from $\mathcal{M}_3$. 
\end{itemize}
We combine Cases i and ii into
\begin{itemize}
\item
{\bf Case 1:} Vectors in $ \mathcal{P}_{\Upsilon^{\le 1}}$ of the form
\begin{equation*}
(s,i,j,k,a,b,t)= (s,-a-b,-p(a-b)-(t-s)+2u,(p-1)(a-b)+(t-s)-u,a,b,t),
\end{equation*} 
where $u\in\{0,1\}$, $1\le s,t\le p$, $a\ge b\ge 0$, $t-s\ge 0$ if $a-b=0$, $u=0$ if $t-s=0$ and $a-b=0$, and $t-s\ge 2-p$ if $u=1$ and $a-b=1$. The set of these vectors equals $\mathcal{S}_1$.
\end{itemize}
The next case we examine is that of vectors $(s,i,j,k,a,b,t)\in \mathcal{P}_{\Upsilon^{\le 1}}$ coming from vectors $(s,j_0,k_0,a,b,t)\in \mathcal{M}_2$.
\begin{itemize}
\item
{\bf Case 2:} The next case is  i.e. vectors $(s,j_0,k_0,a,b,t)\in \mathbb{Z}^6$ 
with $1\le s,t\le p$, $j_0=k_0=0$, $t+s=p+1$, $a,b\ge 0$, $a=b-1$, where we exclude vectors of the form $(p,0,0,a,a+1,1)$ with $a\ge 0$. It follows that the vectors under consideration are of the form
$$
(s,0,0,a,a+1,p+1-s),
$$  
where $1\le s\le p-1$ and $a\ge 0$. Under the operations in \eqref{ijkconvert}, said vectors transform into elements of $\mathcal{P}_{\Upsilon^{\le 1}}$ of the form
$$
(s,i,j,k,a,b,t)=(s,-2a-1,1,0,a,a+1,p+1-s),
$$
where $1\le s\le p-1$ and $a\ge 0$. The set of these vectors equals $\mathcal{S}_2$.
\end{itemize}
Cases 1 and 2 produce all vectors in $\mathcal{P}_{\Upsilon^{\le 0}}$ (in other words, $\mathcal{P}_{\Upsilon^{\le 0}}=\mathcal{S}_1\cup \mathcal{S}_2$). To get the set $\mathcal{P}_{\Upsilon^{\le 1}}$, we need to join the following vectors.\medskip
\begin{itemize}
\item
{\bf Case 3:} Vectors 
$$
(s,i,j,k,a,b,t)=(s,1,1,0,0,0,p+1-s)\in \mathcal{P}_{\Upsilon^{\le 1}},
$$
where $1\le s\le p$. The set of these vectors equals $\mathcal{S}_3$.\end{itemize}

Thus, we have proved that $\mathcal{P}_{\Upsilon^{\le 1}}=\mathcal{S}_1\cup \mathcal{S}_2\cup \mathcal{S}_3$. The sets $\mathcal{S}_1$ and $\mathcal{S}_2$ are disjoint since the vectors in $\mathcal{S}_1$ satisfy $a\ge b$ and those in $\mathcal{S}_2$ satisfy $b=a+1$. The sets $\mathcal{S}_2$ and $\mathcal{S}_3$ are disjoint since the vectors
in $\mathcal{S}_2$ satisfy $b=a+1$ and those in $\mathcal{S}_3$ satisfy $a=0=b$. The sets $\mathcal{S}_1$ and $\mathcal{S}_3$ are disjoint since the vectors in
$\mathcal{S}_1$ satisfy $i\le 0$ and those in $\mathcal{S}_3$ satisfy $i=1$.
\end{proof} 

\section{Analysis of basis elements for $\Ext^k(\Delta_m,\Delta_\ell)$}\label{kbasis}

Our goal is to formulate an algorithm to determine the dimension of the space $\Ext^k(\Delta_m,\Delta_\ell)$. To this end, we give an explicit description of the elements 
$\left({\bf v}_1,{\bf v}_2,...,{\bf v}_q\right)$
of $\mathcal{B}^k(m,\ell)$ as defined in Definition \ref{longdef}. We recall that by Proposition \ref{pupsilon}, ${\bf v}_1,{\bf v}_2,...,{\bf v}_q\in \mathcal{S}_1\cup \mathcal{S}_2\cup \mathcal{S}_3$, and
the sets $\mathcal{S}_1,\mathcal{S}_2,\mathcal{S}_3$ are disjoint. Thus the definition below is meaningful.

\begin{definition}
We say that the $q$-tuple of vectors $\left({\bf v}_1,{\bf v}_2,...,{\bf v}_q\right)$ belongs to case $(x_1,x_2,...,x_q)$ if ${\bf v}_g\in \mathcal{S}_{x_g}$ with $x_g\in \{1,2,3\}$ for all $g\in \{1,...,q\}$. If several adjacent $x_g$ take the same value, we will also say that $\left({\bf v}_1,{\bf v}_2,...,{\bf v}_q\right)$ belongs to case $(x_1^{h_1}x_{h_1+1}^{h_2}, \dots )$ to mean that $x_{h_1}=x_{h_1-1}= \cdots =x_{2}=x_{1}$, $x_{h_1+h_2}=x_{h_1+h_2-1}= \cdots =x_{h_1+2}=x_{h_1+1}$, etc.
\end{definition}

Moreover, throughout the following, we stick to the following conventions:

\begin{notation}{\rm
\begin{itemize}
\item $g$ denotes a natural number such that $1\le g\le q$.
\item For all $g\in \{1,...,q\}$, we set
\begin{equation} \label{wg}
w_g:=t_g-s_g.
\end{equation}
\item For all $g\in \{1,...,q\}$, we assume that
\begin{equation} \label{ug}
u_g\in \{0,1\}.
\end{equation}
\item We set 
\begin{equation} \label{u0w0}
u_0=w_0=c_0:=0.
\end{equation}
\item We define
\begin{equation} \label{Wf}
W_f:= 
\begin{cases} w_0+...+w_f & \mbox{ if } p\ge 3\\ w_f & \mbox{ if } p=2 \end{cases}
\end{equation} 
for $f\in \{1,...,q\}$.
\end{itemize}
}
\end{notation}

As the lemma below shows, only a very restricted set of cases can occur.

\begin{lemma} \label{caseslemma}
A $q$-tuple of vectors $\left({\bf v}_1,{\bf v}_2,...,{\bf v}_q\right)\in \mathcal{B}^k(m,\ell)$ belongs either to case $(1^q)$ or to case $(1^h, 2,3^{q-h-1})$ or to case $(1^h3^{q-h})$ for $1 \leq h \leq q-1$.
\end{lemma}

\begin{proof}
In the definitions of $\mathcal{S}_1$ and $\mathcal{S}_2$, we have $i\le 0$, and in the definitions of $\mathcal{S}_2$ and $\mathcal{S}_3$, we have $j=1$. Since $i_g=j_{g-1}$, we deduce that a vector in $\mathcal{S}_2 \cup \mathcal{S}_3$ can only be followed by a vector in $\mathcal{S}_3$. Moreover, ${\bf v}_1\in \mathcal{S}_1$ since $i_1=j_0=0$. Therefore, only cases of the form $(1^q)$  or  $(1^h, 2,3^{q-h-1})$ or  $(1^h3^{q-h})$  for $1 \leq h \leq q-1$ can occur. 
\end{proof}

Now we formulate a theorem which describes the $q$-tuples belonging to the cases in Lemma \ref{caseslemma} explicitly, in dependence only on two sets of parameters $c_g$ and $u_g$ for $g \in \{1,\dots q\}$. In the pictorial descriptions, the $c_g$ count (leftwards) which column in \eqref{picture} we are in, with the column containing $\ffc$ being labelled by $0$, and in the picture of Example \ref{ex1} the value of $u_g$ can be seen as describing with of the two northwest to southeast diagonals the basis element belongs to, $u_g=0$ corresponding to the upper and $u_g=1$ corresponding to the lower diagonal respectively.

\begin{proposition} \label{longtheorem}
Assume that $\left({\bf v}_1,{\bf v}_2,...,{\bf v}_q\right)\in \mathcal{B}^k(m,\ell)$.
\begin{enumerate}[$($i$)$]
\item\label{long1} If $\left({\bf v}_1,{\bf v}_2,...,{\bf v}_q\right)$ belongs to case $(1^q)$, then, setting $h:=q$,  we have 
\begin{equation} \label{vg}
\begin{split}
{\bf v}_g= & \left(s_g,-pc_{g-1}-w_{g-1}+2u_{g-1},-pc_g-w_g+2u_g,(p-1)c_g+w_g-u_g,\right.\\ &
\left.(c_g+pc_{g-1}+w_{g-1}-2u_{g-1})/2,(-c_g+pc_{g-1}+w_{g-1}-2u_{g-1})/2,t_g\right)\\ 
\end{split}
\end{equation}
for $g\in\{1,2,...,h\}$, where we assume that
\begin{equation} \label{condi2}
\begin{split}
& u_g\in\{0,1\},\ 1\le s_g,t_g\le p,\ c_g\equiv W_{g-1} \bmod{2} \mbox{ if } g\in \{1,...,h\},\\
& (2u_g-w_{g})/p \le c_g\le pc_{g-1}+w_{g-1}-2u_{g-1} \mbox{ if } g\in \{1,...,h-1\},\\  
& (2u_h-w_{h})/p \le c_h\le pc_{h-1}+w_{h-1}-2u_{h-1} \mbox{ or } c_h=0,u_h=1=w_h.
\end{split}
\end{equation}
\item\label{long2} If $\left({\bf v}_1,{\bf v}_2,...,{\bf v}_q\right)$ belongs to case $(1^h, 2,3^{q-h-1})$ with $h\in \{1,...,q-1\}$, then
\begin{equation*} \label{vg2}
\begin{split}
{\bf v}_g= & \left(s_g,-pc_{g-1}-w_{g-1}+2u_{g-1},-pc_g-w_g+2u_g,(p-1)c_g+w_g-u_g,\right.\\ &
\left.(c_g+pc_{g-1}+w_{g-1}-2u_{g-1})/2,(-c_g+pc_{g-1}+w_{g-1}-2u_{g-1})/2,t_g\right)\\ & \mbox{ if } g\in\{1,2,...,h\},\\
{\bf v}_{h+1}= &  \left(s_{h+1},-pc_{h}-w_{h}+2u_{h},1,0,(pc_{h}+w_{h}-2u_{h}-1)/2,\right.\\ &\left.(pc_{h}+w_{h}-2u_{h}+1)/2, 1+p-s_{h+1}\right),\\
{\bf v}_g= & (s_g,1,1,0,0,0,p+1-s_g) \mbox{ if } g\in \{h+2,...,q\}, 
\end{split}
\end{equation*}
where we assume that
\begin{equation*} \label{condi3}
\begin{split}
& u_g\in\{0,1\},\ 1\le s_g,t_g\le p,\ c_g\equiv W_{g-1} \bmod{2} \mbox{ if } g\in\{1,...,h\},\\
& (2u_g-w_g)/p \le c_g\le pc_{g-1}+w_{g-1}-2u_{g-1} \mbox{ if } g\in\{1,...,h\},\\
& W_{h} \equiv 1 \bmod{2},\ 1\le s_{h+1}\le p-1, \ 1\le s_g\le p \mbox{ if } g\in \{h+2,...,q\},\\
& t_g=p+1-s_g \mbox{ if } g\in\{h+1,...,q\}.
\end{split}
\end{equation*}
\item\label{long3} If $\left({\bf v}_1,{\bf v}_2,...,{\bf v}_q\right)$ belongs to  case $(1^h3^{q-h})$ with $h\in \{1,...,q-1\}$, then
\begin{equation*} \label{vg4}
\begin{split}
{\bf v}_g= & \left(s_g,-pc_{g-1}-w_{g-1}+2u_{g-1},-pc_g-w_g+2u_g,(p-1)c_g+w_g-u_g,\right.\\ &
\left.(c_g+pc_{g-1}+w_{g-1}-2u_{g-1})/2,(-c_g+pc_{g-1}+w_{g-1}-2u_{g-1})/2,t_g\right)\\ & \mbox{ if } g\in\{1,2,...,h\},\\
{\bf v}_g= & (s_g,1,1,0,0,0,p+1-s_g) \mbox{ if } g\in \{h+1,...,q\}, 
\end{split}
\end{equation*}
where we assume that
\begin{equation*} \label{condi4}
\begin{split}
& u_g\in\{0,1\},\ 1\le s_g,t_g\le p,\ c_g\equiv W_{g-1} \bmod{2} \mbox{ if } g\in\{1,...,h-1\},\\
& (2u_{g}-w_{g})/p\le c_g\le pc_{g-1}+w_{g-1}-2u_{g-1} \mbox{ if } g\in\{1,...,h-1\},\\
& c_h=0,\ u_h=1=w_h,\ 1\le s_{h}\le p-1,\ W_{h-1} \equiv 0 \bmod{2},\\ 
&  1\le s_g\le p \mbox{ and } t_g=p+1-s_g \mbox{ if } g\in \{h+1,...,q\}.
\end{split}
\end{equation*}
\end{enumerate}
\end{proposition}

\begin{proof}
\eqref{long1} Let $g\in\{1,2,...,q\}$. Then we have $-a_g-b_g=i_g=j_{g-1}$ and thus 
$$
b_g=-j_{g-1}-a_g,
$$ 
and therefore we require that 
$a_g\ge b_g=-j_{g-1}-a_g \ge 0$ 
which is equivalent to
\begin{equation} \label{acond}
-j_{g-1} \ge a_g\ge -j_{g-1}/2.
\end{equation}
Further, it follows that
$$
j_g=-pc_g-w_g+2u_g
$$ 
and 
$$
k_g=(p-1)c_g+w_g-u_g,
$$
where
$$
c_g:=a_g-b_g=2a_g+j_{g-1}.
$$
Similarly,
$$
j_{g-1}=-pc_{g-1}-w_{g-1}+2u_{g-1},
$$
where we recall \eqref{u0w0}.
Moreover, we observe that \eqref{acond} is equivalent to 
$$
0\le c_g \le -j_{g-1}=pc_{g-1}+w_{g-1}-2u_{g-1}.
$$
In particular,
$$
(2u_{g-1}-w_{g-1})/p\le c_{g-1} \mbox{ if } g\in \{1,...,q\}
$$
and hence
\begin{equation} \label{condi1}
(2u_{g}-w_{g})/p\le c_{g} \mbox{ if } g\in \{1,...,q-1\}.
\end{equation}

We observe that if $g\in \{1,...,q-1\}$, then \eqref{condi1} implies the conditions 
\begin{equation} \label{Case1condi}
\begin{cases}
w_{g}=t_g-s_g\ge 0 & \mbox{ if } c_{g}=a_{g}-b_{g}=0,\\
u_g=0 & \mbox{ if } w_g=0 \mbox{ and } c_g=a_g-b_g=0,\\
w_{g}=t_g-s_g\ge 2-p & \mbox{ if } u_{g}=1 \mbox{ and } c_{g}=a_{g}-b_{g}=1 
\end{cases} 
\end{equation}
in the definition of $\mathcal{S}_1$. If $g=q$, then \eqref{Case1condi} is equivalent to
$$
(2u_{g}-w_{g})/p\le c_{g} \mbox{ or } c_{g}=0,u_g=1=w_g.
$$

We further express $a_g$ and $b_g$ in the form
$$
a_g=(c_g-j_{g-1})/2=(c_g+pc_{g-1}+w_{g-1}-2u_{g-1})/2 
$$
and
$$
b_g=(-c_g-j_{g-1})/2=(-c_g+pc_{g-1}+w_{g-1}-2u_{g-1})/2,
$$
where we assume that 
$$
\begin{cases}
c_{g-1}+w_{g-1} \equiv pc_{g-1}+w_{g-1} \equiv  c_g \bmod{2} & \mbox{ if } p\ge 3,\\
w_{g-1} \equiv 2c_{g-1}+w_{g-1} \equiv c_g \bmod{2} & \mbox{ if } p=2.
\end{cases}
$$
Iterating the last congruence, we obtain the condition
$$
\begin{cases}
c_g\equiv (w_0+...+w_{g-1}) \bmod{2}  & \mbox{ if } p\ge 3,\\
c_g\equiv w_{g-1} \bmod{2} & \mbox{ if } p=2.
\end{cases}
$$

Putting the above in a closed form, and considering the conditions in the definition of $\mathcal{S}_1$, we deduce the claim.

\eqref{long2} Here we have ${\bf v}_1,...,{\bf v}_h\in \mathcal{S}_1$, ${\bf v}_{h+1}\in \mathcal{S}_2$, ${\bf v}_{h+2},...,{\bf v}_q\in \mathcal{S}_3$ for some $h\in \{1,...,q-1\}$.  By the considerations in \eqref{long1}, the vectors ${\bf v}_1,...,{\bf v}_{h}$ satisfy \eqref{vg} under the conditions in \eqref{condi2}. Since ${\bf v}_{h+1}\in \mathcal{S}_2$, ${\bf v}_{h+2},...,{\bf v}_q \in \mathcal{S}_3$, it follows that
\begin{equation*}
\begin{split}
{\bf v}_{h+1}= & (s_{h+1},-2a_{h+1}-1,1,0,a_{h+1},a_{h+1}+1,p+1-s_{h+1}),\\
{\bf v}_{h+2}= & (s_{h+2},1,1,0,0,0,p+1-s_{h+2}),\\
..., &\\
{\bf v}_q= & (s_q,1,1,0,0,0,p+1-s_q),
\end{split}
\end{equation*}
where $1\le s_{h+1}\le p-1$, $a_{h+1}\ge 0$ and $1\le s_g\le p$ if $g\in \{h+2,...,q\}$. Further,
$$
-2a_{h+1}-1=j_{h}=-pc_{h}-w_{h}+2u_{h}
$$
and hence
$$
a_{h+1}=(pc_{h}+w_{h}-2u_{h}-1)/2.
$$
Thus we require that
\begin{equation} \label{congru}
\begin{cases}
w_0+...+w_{h-1}+w_{h}\equiv c_{h}+w_{h}\equiv pc_{h}+w_{h}\equiv 1 \bmod{2} & \mbox{ if } p\ge 3,\\
w_h \equiv 2c_h+w_h \equiv 1 \bmod{2} & \mbox{ if } p=2
\end{cases}
\end{equation}
and
\begin{equation}  \label{arti}
pc_{h}+w_{h}-2u_{h}-1\ge 0.
\end{equation}
Ignoring \eqref{arti} for a moment, we can conclude from \eqref{congru} alone that
$$
pc_{h}+w_{h}-2u_{h}\not=0.
$$
Hence, under condition \eqref{congru}, \eqref{arti} is equivalent to
$$
pc_{h}+w_{h}-2u_{h}\ge 0
$$
and hence to 
\begin{equation*} \label{true}
(2u_h-w_h)/p \le c_h.
\end{equation*}
Therefore, in this situation, we can drop the case 
$$
c_h=0,u_h=1=w_h.
$$
in \eqref{condi2}. Putting the above in a closed form, we deduce the claim.

\eqref{long3} Here we have ${\bf v}_1,...,{\bf v}_h\in \mathcal{S}_1$ and ${\bf v}_{h+1},...,{\bf v}_q\in \mathcal{S}_3$ for some $h\in \{1,...,q-1\}$. By the considerations in \eqref{long1}, the vectors ${\bf v}_1,...,{\bf v}_{h}$ satisfy \eqref{vg} under the conditions in \eqref{condi2}. Moreover,
\begin{equation*} \label{vg3}
{\bf v}_g=(s_g,1,1,0,0,0,p+1-s_g) \mbox{ with } 1\le s_g\le p \mbox{ if } g\in \{h+1,...,q\},  
\end{equation*}
and
\begin{equation} \label{this}
-pc_h-w_h+2u_h=j_h=1.
\end{equation}
This is incompatible with $(2u_h-w_h)/p\le c_h$ and hence $c_h=0,u_h=1=w_h$ by \eqref{condi2} and \eqref{this}. The claim follows. 
\end{proof}

\section{Reduction to $A_{h,k}^{w_1,...,w_h}$} \label{compdim}
In this section, we provide formula for the dimension
\begin{equation*} 
	\dim \Ext^k\left(\Delta_m,\Delta_\ell\right)=\sharp\mathcal{B}^k(m,\ell)
\end{equation*}
which only depends on the cardinality of certain sets which we now introduce.

\begin{definition} \label{Adef1} For $h\ge 1$ set
\begin{equation*} 
\begin{split}
A_{h,k}^{w_1,...,w_h}:= \sharp\Big\{ & (u_1,...,u_{h},c_1,...,c_h)\in \{0,1\}^h\times \mathbb{Z}^h\ :\\  & \sum\limits_{g=1}^h \left((p-1)c_g+w_g-u_g\right)=k, \ 
 c_g\equiv W_{g-1} \bmod{2} \\ & \mbox{\rm\ and } (2u_g-w_g)/p\le c_g\le pc_{g-1}+w_{g-1}-2u_{g-1}\\ & \mbox{\rm\ for } g\in\{1,...,h\}\Big\}.
\end{split}
\end{equation*}
\end{definition}

Now we can express the quantity in question in the following form. 

\begin{proposition} \label{dimtheorem}
We have
\begin{equation} \label{dimformel}
\dim \Ext^k\left(\Delta_m,\Delta_\ell\right)=D_1+D_2+D_3+D_4,
\end{equation}
where
\begin{eqnarray} 
D_1 &:= & A_{q,k}^{w_1,...,w_q}, \label{dimform1} \\
D_2 &:= & \sum\limits_{\substack{h=1\\ W_{h} \equiv 0\bmod{2}\\ w_{h+1}=1\\ t_{h+2}=p+1-s_{h+2},...,t_q=p+1-s_q}}^{q-1} A_{h,k}^{w_1,...,w_{h}}, \label{dimform2} \\ 
D_3 &:= & \begin{cases} 1 & \mbox{\rm \ if } k=0,\ w_1=1,\ t_2=p+1-s_2,...,\ t_q=p+1-s_q,\\ 0 & \mbox{\rm \ otherwise,} \end{cases}\label{dimform3} \\
D_4 &:=& \sum\limits_{\substack{h=1\\ W_h\equiv 1\bmod{2}\\ s_{h+1}\not=p\\ t_{h+1}=p+1-s_{h+1},...,t_q=p+1-s_q}}^{q-1} A_{h,k}^{w_1,...,w_h}. \label{dimform4}
\end{eqnarray}
\end{proposition}

\begin{proof}
This follows from Lemma \ref{caseslemma} and Proposition \ref{longtheorem}.
In \eqref{dimformel}, the term $D_1$ comes from Case 1,...,1 with $(2u_q-w_q)/p\le c_q\le pc_{q-1}+w_{q-1}-2u_{q-1}$, the term $D_2$ comes from Cases 1,...,1 and 1,...,1,3,...,3 (after reparametrizing $h\rightarrow h+1$) with $c_{h+1}=0$ and $u_{h+1}=1=w_{h+1}$, the first $h+1$ of the $q$ vectors belonging to $\mathcal{S}_1$, the term $D_3$ comes from Case 1,3,...,3, and the term $D_4$ comes from Case 1,...,1,2,3,...,3. 
\end{proof}

Before we proceed,  we make the following observation, which we check combinatorially, but which algebraically simply reflects that $\bfw_q \cong e\bfw_{q+1}e$ for the idempotent $e$ in $\bfw_{q+1}$ picking out the first $p^q$ simple modules.

\begin{lemma} \label{ind} The value of $\dim \Ext^k(\Delta_m,\Delta_\ell)$ is independent of $q$. 
\end{lemma}

\begin{proof}
In Definition \ref{longdef}, for given $q$, we define $q$-tuples $(s_1,...,s_q)$ and $(t_1,...,t_q)$ associated to $m$ and $\ell$, which then give rise to a $q$-tuple $(w_1,...,w_q)$. Now let $\tilde{q}>q$. Then the corresponding $\tilde{q}$-tuples associated to $m$ and $\ell$ become $(1,...,1,s_1,...,s_q)$, $(1,...,1,t_1,...,t_q)$ and $(0,...,0,w_1,...,w_q)$. Note that
$$
A^{0,...,0,w_1,...,w_q}_{\tilde{q},k}=A^{w_1,...,w_q}_{q,k}
$$
using Definition {\rm \ref{Adef1}}. Hence, the terms $D_1$ and $D_4$ defined in \eqref{dimform1} and \eqref{dimform4} stay the same if $q$ is replaced by $\tilde{q}$, and $D_2+D_3$ also stays the same upon noting that
$$
A^{0,...,0}_{\tilde{q}-q,0}=1.
$$
Hence, by Proposition \ref{dimtheorem}, $\Ext^k(\Delta_m,\Delta_\ell)$ stays unchanged as well.
\end{proof}

Using Proposition {\rm \ref{dimtheorem}}, we have reduced the problem to determining $A_{h,k}^{w_1,...,w_h}$. To investigate this quantity, it is natural to break it up according to the values of $u_1,...,u_h$. Therefore, we introduce the following related quantity.

\begin{definition} \label{Bdef} For $l\in \mathbb{Z}$, $h\in \mathbb{N}$ and $v_1,...,v_h\in \mathbb{Z}$ set
\begin{equation*} 
\begin{split}
B_{h,l}^{v_1,...,v_h}:= \sharp\Big\{ & (c_1,...,c_h)\in \mathbb{Z}^h\ :\  \sum\limits_{g=1}^h \left((p-1)c_g+v_g\right)=l, \\ 
& c_g\equiv V_{g-1} \bmod{2} \mbox{\rm\ and } -v_g/p\le c_g\le pc_{g-1}+v_{g-1}\\ & \mbox{\rm\ for } g\in\{1,...,h\}\Big\},
\end{split}
\end{equation*}
where $c_0=v_0=V_0:=0$ and
\begin{equation} \label{f}
V_f:= 
\begin{cases} v_1+...+v_f & \mbox{ if } p\ge 3,\\ v_f & \mbox{ if } p=2 \end{cases}
\end{equation} 
for $f\in \{1,...,h\}$.
\end{definition}

Now the following is obvious.

\begin{lemma} \label{A1B} We have
\begin{equation} \label{ABeq}
A_{h,k}^{w_1,...,w_h}=\sum\limits_{(u_1,...,u_h)\in \{0,1\}^h} B_{h,k-(u_1+...+u_h)}^{w_1-2u_1,...,w_h-2u_h}.
\end{equation}
\end{lemma}

In the next section, we shall develop recursive formulas for $B_{h,l}^{v_1,...,v_h}$ and $A_{h,k}^{w_1,...,w_h}$. These formulas will later be used for the computation of the dimension of $\Ext^k(\Delta_m,\Delta_\ell)$ by a computer program and an investigation of the dimension growth.

\section{Recursive formulas} \label{secrecursions}
Throughout the following, we use the following notations.

\begin{definition} \label{deltas}
For $a\in \mathbb{N}$ and $b\in \mathbb{Z}$ define
\begin{equation*}
\delta(b;a) = \begin{cases} 1 & \mbox{\rm if } a|b, \\ 0 & \mbox{\rm otherwise.}\end{cases}
\end{equation*}
\end{definition}

We start by working out the following recursion.

\begin{proposition} \label{recurs}
Let $l\in \mathbb{Z}$, $h\in \mathbb{N}$ and $v_1,...,v_h\in \mathbb{Z}$. Then
\begin{enumerate}[$($i$)$]
\item\label{proprec1} $B_{h,l}^{v_1,v_2,...,v_h}=0 \ \mbox{\rm if } v_1<0$;
\item\label{proprec2} $B_{h,l}^{0,v_2,...,v_h}=B_{h-1,l}^{v_2,...,v_h} \ \mbox{\rm if } h>1$;
\item\label{proprec3} $B_{1,l}^{v_1} = \begin{cases} 1 & \mbox{\rm if } l\equiv v_1\bmod{2(p-1)}, \ v_1/p\le l\le v_1,\\ 0 & \mbox{\rm otherwise,} \end{cases}$;
\item\label{proprec4} $B_{h,l}^{v_1,...,v_h}=B_{h,l}^{v_1-1,v_2+p,v_3,...,v_h}+\delta(v_1;2p)\cdot B_{h-1,l-v_1/p}^{v_2,v_3,...,v_h} \quad \mbox{\rm if } v_1>0, h>1$.
\end{enumerate}
\end{proposition}

\begin{proof}
The equations in parts \eqref{proprec1} and \eqref{proprec2} follow immediately from Definition \ref{Bdef}. 

By the same definition,
$$
B_{1,l}^{v_1} = \begin{cases} 1 & \mbox{\rm if } l-v_1\equiv 0 \bmod{p-1}, \ (l-v_1)/(p-1)\equiv 0\bmod{2}, \\ & -v_1/p\le (l-v_1)/(p-1)\le 0,\\ 0 & \mbox{\rm otherwise,} \end{cases}
$$
which implies the equation in part \eqref{proprec3}.

To prove part \eqref{proprec4}, we rewrite $B_{h,l}^{v_1-1,v_2+p,v_3,...,v_h}$ using Definition \ref{Bdef} again.
\begin{equation*} 
\begin{split}
& B_{h,l}^{v_1-1,v_2+p,v_3,...,v_h}\\ = \sharp\Big\{ & (c_1',...,c_h')\in \mathbb{Z}^h\ :\  \\ & ((p-1)c_1'+v_1-1)+((p-1)c_2'+v_2+p)+\sum\limits_{g=3}^h \left((p-1)c_g'+v_g\right)=l, \\ 
& c_1'\equiv 0\bmod{2},\ c_2'\equiv V_1-1 \bmod{2},\  c_g'\equiv V_{g-1} \bmod{2}\ \mbox{\rm for } g\in \{3,...,h\}, \\ & -(v_1-1)/p\le c_1'\le 0, \ -(v_2+p)/p\le c_2'\le pc_1'+v_1-1, \\ &
-v_g/p\le c_g'\le pc_{g-1}'+v_{g-1} \mbox{\rm\ for } g\in\{3,...,h\}\Big\}\\
=\sharp\Big\{ & (c_1,...,c_h)\in\mathbb{Z}^h\ :\ \sum\limits_{g=1}^h \left((p-1)c_g+v_g\right)=l,\\ & 
c_g\equiv V_{g-1} \bmod{2} \mbox{\rm\ for } g\in \{1,...,h\},\ -(v_1-1)/p\le c_1\le 0,\\ & -v_g/p\le c_g\le pc_{g-1}+v_{g-1} \mbox{\rm\ for } g\in\{2,...,h\}\Big\}
\end{split}
\end{equation*}
\begin{equation} \label{sotoll}
\begin{split}
= B&_{h,l}^{v_1,...,v_2}-\delta(v_1;2p)\cdot \sharp\Big\{ (c_1,...,c_h)\in\mathbb{Z}^h\ :\ \sum\limits_{g=1}^h \left((p-1)c_g+v_g\right)=l,\\ & c_1\equiv 0\bmod{2}, \mbox{ and } c_1=-v_1/p,\\ &
c_g\equiv V_{g-1} \bmod{2} \mbox{\rm\ \ and } -v_g/p\le c_g\le pc_{g-1}+v_{g-1} \mbox{\rm\ for } g\in\{2,...,h\}\Big\},
\end{split}
\end{equation}
where the second line comes from the changes of variables $c_1=c_1',c_2=c_2'+1,c_3=c_3',...,c_h=c_h'$. We observe that $(p-1)c_1+v_1=v_1/p$ and $pc_1+v_1=0$ if $c_1=-v_1/p$.  Now it follows from Definition \ref{Bdef} that the term in \eqref{sotoll} equals 
$$
 B_{h,l}^{v_1,...,v_h}-\delta(v_1;2p)\cdot B_{h-1,l-v_1/p}^{v_2,v_3,...,v_h},
$$
which implies the claim in part \eqref{proprec4}.
\end{proof}

\begin{remark}\label{rem1}  Parts \eqref{proprec1} and \eqref{proprec3} of Proposition {\rm \ref{recurs}} imply that $B_{1,l}^{v_1}=0$ if $l<0$ or $v_1<0$.  
\end{remark}

By iterating Proposition \ref{recurs}\eqref{proprec4}, we obtain the following.

\begin{theorem} \label{recu2} Let $l\in \mathbb{Z}$, $h\in \mathbb{N}$ and $v_1,...,v_h\in \mathbb{Z}$. Suppose that $h>1$. Then
\begin{equation} \label{Brec1}
B_{h,l}^{v_1,...,v_h}=\sum\limits_{0\le d\le v_1/(2p)} B_{h-1,l-2d}^{v_2+p(v_1-2dp),v_3,...,v_h}.
\end{equation}
\end{theorem}

\begin{proof} If $v_1<0$, then this result holds by part \eqref{proprec1} of Proposition \ref{recurs}. Otherwise,
we apply part \eqref{proprec4} of Proposition \ref{recurs} $v_1$ times and then part \eqref{proprec2}. 
\end{proof}

Combining Lemma \ref{A1B}, Proposition \ref{recurs}\eqref{proprec3} and Theorem \ref{recu2}, we obtain the following result on $A_{h,k}^{w_1,...,w_h}$. 

\begin{cor} \label{recu2A} Let $k\in \mathbb{Z}$, $h\in \mathbb{Z}$ and $w_1,...,w_h\in \mathbb{Z}$. Then 
\begin{enumerate}[$($i$)$]
\item\label{recu2A1} $A_{1,k}^{w_1} = \begin{cases} 1 & \mbox{\rm if } k+u_1\equiv w_1\bmod{2(p-1)} \ \mbox{\rm and }\\ & (w_1+2(p-1)u_1)/p\le k+u_1\le w_1 \mbox{ for } u_1=0 \mbox{ or 
} u_1=1,\\ 0 & \mbox{\rm otherwise,} \end{cases}$;
\item\label{recu2A2} $A_{h,k}^{w_1,...,w_h}=\sum\limits_{u_1=0}^1\ \sum\limits_{0\le d\le (w_1-2u_1)/(2p)} A_{h-1,k-u_1-2d}^{w_2+p(w_1-2u_1-2dp),w_3,...,w_h} \ \mbox{\rm if } h>1.$
\end{enumerate}
\end{cor}

The above Corollary \ref{recu2A} will be the basis for an explicit computation of the dimension by a computer program. In the next section we shall derive an explicit expression for $B_{h,l}^{v_1,...,v_h}$ which will later be used in deriving lower and upper bounds for the dimension.

\section{Explicit formulas in term of a partition function} \label{redtopartitions}

In this section, we express the quantities $B_{h,l}^{v_1,...,v_h}$ in terms of a certain partition function, which we now define.

\begin{definition} \label{qQdef} For given integers $D$ and $d$ let $q_p(D,d)$ be the number of simultaneous respresentations of $D$ and $d$ in the form 
$$
D=n_0+n_1p+n_2p^2+... \quad \mbox{and} \quad d=n_0+n_1+n_2+...,
$$
where $n_0,n_1,n_2,..$ is a sequence of non-negative integers. 
\end{definition}


Now, by iterating \eqref{Brec1}, we relate $B_{h,l}^{v_1,...,v_h}$ to $B_{1,\tilde{l}}^{\tilde{v}}$ for certain $\tilde{l}$ and $\tilde{v}$. 

\begin{lemma} \label{recu3} Let $l\in \mathbb{Z}$, $h\in \mathbb{N}$ and $v_1,...,v_h\in \mathbb{Z}$. Suppose that $h>1$ and $v_2,...,v_h\le p-1$.
Set $$
V:=v_h+v_{h-1}p+...+v_1p^{h-1} \quad \hbox{ and } \quad D:=d_{h-1}+pd_{h-2}+...+p^{h-2}d_1.$$

 Then
\begin{equation*}
B_{h,l}^{v_1,...,v_h}= \sum\limits_{d_1\ge 0} \cdots \sum\limits_{d_{h-1}\ge 0} 
B_{1,l-2(d_1+...+d_{h-1})}^{V-2p^2D}.
\end{equation*}
\end{lemma}

\begin{proof} Applying the recursion in Lemma \ref{recu2} $h-1$ times, we get 
\begin{equation} \label{applyrec}
\begin{split}
B&_{h,l}^{v_1,...,v_h}\\&= \sum\limits_{0\le d_1\le \frac{v_1}{2p}}\ \sum\limits_{0\le d_2\le \frac{v_2+pv_1-2p^2d_1}{2p}} \cdots
\sum\limits_{0\le d_{h-1}\le \frac{V-2p^2D}{2p}}
B_{1,l-2(d_1+...+d_{h-1})}^{V-2p^2D}.
\end{split}
\end{equation}
The following argument shows that the summation conditions on $d_1,...,d_{h-1}$ above can be replaced by simply $d_1\ge 0$,...,$d_{h-1}\ge 0$, proving the claim. By part \eqref{proprec1} of Proposition \ref{recurs},  the summand in \eqref{applyrec} is 0 unless
\begin{equation} \label{schab}
V-2p^2D\ge 0.
\end{equation}
If $d_1\ge 0$,...,$d_{h-1}\ge 0$, then \eqref{schab} implies
\begin{equation}
2p\left(d_r+pd_{r-1}+...+p^{r-1}d_1\right)\le \frac{v_h+pv_{h-1}+...+p^{h-1}v_1}{p^{h-r}}
\end{equation}
for $1\le r\le h-1$. Since
\begin{equation*}
\begin{split}
\frac{v_h+pv_{h-1}+...p^{h-r-1}v_{r+1}}{p^{h-r}}\le & \frac{(p-1)+(p-1)p+...+(p-1)p^{h-r-1}}{p^{h-r}}\\
= & \frac{p^{h-r}-1}{p^{h-r}}<1
\end{split}
\end{equation*}
and 
$$
\frac{p^{h-r}v_{r}+p^{h-r+1}v_{r-1}+...+p^{h-1}v_1}{p^{h-r}}=v_r+pv_{r-1}+...+p^{r-1}v_1\in \mathbb{N},
$$
it follows that
$$
2p\left(d_r+pd_{r-1}+...+p^{r-1}d_1\right)\le  v_r+pv_{r-1}+...+p^{r-1}v_1,
$$
implying 
$$
d_r\le \frac{v_r+pv_{r-1}+...+p^{r-1}v_1-2p^2\left(d_{r-1}+...+p^{r-2}d_1\right)}{2p}= \frac{V-2p^2D}{2p}.
$$
\end{proof}

\begin{remark} \label{rem2} Part \eqref{proprec1} of Proposition {\rm \ref{recurs}}, Remark {\rm \ref{rem1}} and Lemma {\rm \ref{recu3}} imply that $B_{h,l}^{v_1,...,v_h}=0$ if $l<0$ or $v_h+pv_{h-1}+...+p^{h-1}v_1<0$. 
\end{remark}

We prove the following.

\begin{theorem} \label{expB} Let $l\in \mathbb{Z}$, $h\in \mathbb{N}$ and $v_1,...,v_h\in \mathbb{Z}$ and let
$$
V:=v_h+v_{h-1}p+...+v_1p^{h-1}.
$$
Suppose that $h>1$, $v_1,v_2,...,v_h\le p-1$, $V\ge 0$ and $l\ge 0$.
Then
\begin{equation} \label{Bexplicit}
B_{h,l}^{v_1,...,v_h}=\delta(V-l;2(p-1))\cdot \sum\limits_{0\le d\le l/2} \ \sum\limits_{\frac{V-p(l-2d)}{2p^2}\le D\le \frac{V-(l-2d)}{2p^2}} q_p(D,d).
\end{equation}
\end{theorem}

\begin{proof}
As in Lemma \ref{recu3}, set $D:=d_{h-1}+pd_{h-2}+...+p^{h-2}d_1$ and furthermore set $d:=d_{h-1}+d_{h-2}+...+d_1$.
By part \eqref{proprec3} of Proposition \ref{recurs}, we have
\begin{equation} \label{whereisthespring}
B_{1,l-2d}^{V-2p^2D}=\begin{cases} 1 & \mbox{\rm if } l-2d\equiv V-2p^2D\bmod{2(p-1)} \mbox{ and } \\ &  (V-2p^2D)/p\le l-2d\le V-2p^2D,\\ 0 & \mbox{\rm otherwise.} \end{cases}
\end{equation}
We observe that
$$
2d\equiv 2p^2D\bmod{2(p-1)} ,
$$
and thus the congruence condition in \eqref{whereisthespring} turns into 
$$
V\equiv l \bmod{2(p-1)}.
$$
The inequality in \eqref{whereisthespring} can be rewritten in the form
$$
\frac{V-p(l-2d)}{2p^2}\le D\le \frac{V-(l-2d)}{2p^2}. 
$$
Combining the above with Lemma \ref{recu3}, and using Definition \ref{qQdef}, we deduce that
\begin{equation} \label{nunja}
\begin{split}
B&_{h,l}^{v_1,...,v_h}\\&= \delta(V-l;2(p-1)) \sum\limits_{0\le d\le l/2}\, \sum\limits_{\frac{V-p(l-2d)}{2p^2}\le D\le \frac{V-(l-2d)}{2p^2}} \sum\limits_{\substack{d_1,...,d_{h-1}\ge 0\\ d_{h-1}+...+d_1=d\\
d_{h-1}+pd_{h-2}+...+p^{h-2}d_1=D}} 1.
\end{split}
\end{equation}
Now, if $n_0,n_1,n_2,...$ is a sequence of non-negative integers satisfying
$$
D=n_0+n_1p+n_2p^2+... \quad \mbox{and} \quad d=n_0+n_1+n_2+...,
$$
and $D\le V/(2p^2)$, then necessarily $n_r=0$ for $r>h-2$, since 
$$
V\le (p-1)+(p-1)p+...+(p-1)p^{h-1}=p^h-1.
$$ 
It follows that
\begin{equation} \label{klar}
\sum\limits_{\substack{d_1,...,d_{h-1}\ge 0\\ d_{h-1}+...+d_1=d\\ d_{h-1}+pd_{h-2}+...+p^{h-2}d_1=D}} 1=q_p(D,d)
\end{equation}
if $D\le V/(2p^2)$. Combining \eqref{nunja} and \eqref{klar}, the claim follows. 
\end{proof}

\begin{remark} \label{rem3}
If $D<d$, then trivially $q_p(D,d)=0$. Therefore, part \eqref{proprec3} of Proposition {\rm \ref{recurs}}, Lemma $\mbox{\rm \ref{recu3}}$ and \eqref{whereisthespring} imply that $B^{v_1,...,v_h}_{h,l}=0$ if $l>V$.
\end{remark}

The partition function $q_p(D,d)$ is of great interest in its own right. We will investigate it thoroughly in the next section.

\section{Investigation of $q_p(D,d)$}\label{partinvest}
\subsection{Reformulation of the problem and notations} \label{refprob} 
We first rewrite $q_p(D,d)$ in a more convenient form. To this end, we introduce the following quantity which will be investigated in the remainder of this section. 

\begin{definition} \label{rMd} For given integers $M$ and $d$, let $r_p(M,d)$ be the number of representations of the integer $M$ in the form 
$$
M=m_0+m_1p+m_2p^2+...,
$$
where $m_0,m_1,...$ is a sequence of non-negative integers satisfying $d\ge m_0\ge m_1\ge...$.
\end{definition}

We have the following relation between $q_p(D,d)$ and $r_p(M,d)$. 

\begin{lemma} \label{rqrelation} For any two non-negative integers $D$ and $d$, we have
\begin{equation*}
q_p(D,d)=\begin{cases} 0 & \mbox{if } D-d\not\equiv 0 \bmod{p-1}, \\ r_p\left(\frac{D-d}{p-1},d\right) & \mbox{if } D-d\equiv 0\bmod{p-1}. \end{cases}
\end{equation*}
\end{lemma}

\begin{proof}
Taking the difference of the two equations
\begin{eqnarray*}
D &=& n_0+n_1p+n_2p^2+...\\
d&=&n_0+n_1+n_2+...,
\end{eqnarray*}
we get
\begin{eqnarray*}
D-d &=& n_1(p-1)+n_2(p^2-1)+n_3(p^3-1)+...\\ &=& (p-1)\left(n_1+n_2(p+1)+n_3(p^2+p+1)+...\right)\\
&=&(p-1)\left(m_0+m_1p+m_2p^2+...\right),
\end{eqnarray*}
where 
$$
m_j=\sum\limits_{i\ge j+1} n_i.
$$
This implies the claim.
\end{proof}

For the investigation of $r_p(M,d)$, we shall use the following related function.

\begin{definition} \label{rhMd} For given integers $M$, $d$ and $h\ge 0$, let $r_p^h(M,d)$ be the number of representations of $M$ in the form
$$
M=d_0p^h+d_1p^{h-1}+...+d_h,
$$
where $d_0,d_1,...,d_h$ are integers satisfying $0\le d_0\le d_1\le...\le d_h\le d$. 
\end{definition}

The following are obvious consequences of Definitions \ref{rMd} and \ref{rhMd} and will be used throughout this section.

\begin{lemma} \label{trivialobs} The following statements are true for all integers $M$, $d$ and $h\ge 0$. 
\begin{enumerate}[$($i$)$]
\item\label{trivial1} $r_p^h(M,d)\le r_p(M,d)$,
\item\label{trivial2}  $r_p(M,d)=r_p(M,M)$ if $d\ge M$,
\item\label{trivial3}  $r_p(M,d)\le r_p(M,e)$ if $d<e$,
\item\label{trivial4}  $r_p(M,d)\le r_p(M,M)$.
\end{enumerate}
\end{lemma} 

The function we are mainly interested in is defined below.

\begin{definition} \label{13} For $d\in \mathbb{Z}$, we set 
$$
Z_p(d):=\max\limits_{M\in \mathbb{Z}} r_p(M,d).
$$
\end{definition}

It will turn out that $Z_p(d)$ is always finite, i.e. for any given integer $d$, $r_p(M,d)$ stays bounded as $M$ runs through the integers.

In the next subsections, we derive recursions for $r_p(M,d)$ and $r_p^h(M,d)$, which will then be used to bound $Z_p(d)$ from below and above.

\subsection{Recursions for $r_p(M,d)$ and $r_p^h(M,d)$} \label{recn2} We start by establishing a recursion for the function $r_p(M,d)$. 

\begin{proposition} \label{secrec}
Assume that $M$ and $d$ are non-negative integers. Then
\begin{equation} \label{TAL}
r_p(M,d)=\sum\limits_{(M-d)/p\le N\le M/p} r_p(N,M-Np).
\end{equation}
\end{proposition}

\begin{proof}
Using Definition \ref{rMd}, we have
\begin{equation}
\begin{split}
r_p(M,d)= & \sum\limits_{\substack{0\le n_0\le d\\ n_0\equiv M \bmod{p}}} \sharp\{(n_m)_{m\in \mathbb{N}} \ : \  n_0,n_1,n_2,...\in \mathbb{N}\cup \{0\},\\
& \ n_0\ge n_1\ge n_2\ge..., \ n_1p+n_2p^2+...=N-n_0\} \\
= & \sum\limits_{\substack{0\le n_0\le d\\ n_0\equiv M \bmod{p}}} \sharp\{(n_m)_{m\in \mathbb{N}} \ : \  n_0,n_1,n_2,...\in \mathbb{N}\cup \{0\},\\
&  n_0\ge n_1\ge n_2\ge..., \ n_1+n_2p+...=(N-n_0)/p\} \\
= & \sum\limits_{\substack{0\le n_0\le d\\ n_0\equiv M \bmod{p}}} r_p\left(\frac{M-n_0}{p},n_0\right) \\
= &  \sum\limits_{(M-d)/p\le N\le M/p} r_p(N,M-Np),
\end{split}
\end{equation}
establishing the claim.
\end{proof}

From Proposition \ref{secrec}, we deduce the following recursive lower and upper bound for the function $r_p(d)=r_p(d,d)$, which will be useful for the estimation of $Z_p(d)$.

\begin{cor} \label{rdd} Let $d$ be a non-negative integer. Then we have
\begin{equation}
\sum\limits_{N\le d/(p+1)} r_p(N,N)\le r_p(d,d)\le \sum\limits_{N\le d/p} r_p(N,N).
\end{equation}
\end{cor}

\begin{proof}
By Lemma \ref{trivialobs}\eqref{trivial1}, we  have $r_p(N,f)\le r_p(N,N)$ for any $f$, which together with Proposition \ref{secrec} gives the upper bound. Moreover, $r_p(N,d-Np)=r_p(N,N)$
if $d-Np\ge N$ by Lemma \ref{trivialobs}\eqref{trivial2}. The latter is the case if $N\le d/(p+1)$. This together with Proposition \ref{secrec} implies the lower bound. 
\end{proof}

Further, we establish the following recursion for $r_p^h(M,d)$.

\begin{proposition} \label{rhrec}
Suppose that $M,d,h$ are integers and $h\ge 1$. Then we have
\begin{equation} \label{basicrec} 
r_p^h(M,d) = \sum\limits_{0\le f\le d} r_p^{h-1}\left(M-f\cdot \frac{p^{h+1}-1}{p-1},d-f\right)
\end{equation}
\end{proposition}

\begin{proof}
We may write
\begin{equation*} 
\begin{split}
r_p^h(M,d) = & \mathop{\sum\limits_{0\le d_0\le d} \ \sum\limits_{d_0\le d_1\le...\le d_h\le d}}_{d_0p^h+d_1p^{h-1}+...+d_h=M} 1\\
= & \mathop{\sum\limits_{0\le d_0\le d}\  \sum\limits_{0\le f_{1}\le...\le f_h\le d-d_0}}_{\substack{M-d_0\left(p^{h}+p^{h-1}+...+1\right)=
\\ f_{1}p^{h-1}+f_{2}p^{h-2}+...+f_h}} 1\\
= & \mathop{\sum\limits_{0\le f\le d}\  \sum\limits_{0\le f_{1}\le...\le f_h\le d-f}}_{\substack{M-f\cdot \frac{p^{h+1}-1}{p-1}=
\\ f_{1}p^{h-1}+f_{2}p^{h-2}+...+f_h}} 1\\
= & \sum\limits_{0\le f \le d} r_p^{h-1}\left(M-f\cdot \frac{p^{h+1}-1}{p-1},d-f\right).
\end{split}
\end{equation*}
\end{proof}

\subsection{Upper and lower bounds for $r_p(d,d)$} Using Corollary \ref{rdd}, we establish the following explicit lower and upper bounds for $r_p(d,d)$.

\begin{theorem} \label{STfunctions} Define the functions $S_p,T_p:\mathbb{R}\rightarrow \mathbb{R}$ by
$$
S_p(x)=\sum\limits_{n=0}^{\infty} a_nx^n \quad \mbox{and} \quad T_p(x)=\sum\limits_{n=0}^{\infty} b_nx^n,
$$
where we set
$$
a_n=\left(2n! \prod\limits_{j=1}^n \left((p+1)^j+1\right)\right)^{-1} \quad \mbox{and} \quad b_n=\left(n!\prod\limits_{j=1}^n \left(p^j-1\right)\right)^{-1}.
$$
Then
\begin{equation} \label{yeah!}
S_p(d)\le r_p(d,d)\le T_p(d)
\end{equation}
for all integers non-negative $d$. 
\end{theorem}

\begin{proof}
We first note that for $x\ge 0$, $S_p(x)$ and $T_p(x)$ are monotonically increasing since their derivatives are positive there. Moreover, the coefficients $a_n$ and $b_n$ satisfy the recursions
\begin{equation} \label{abcoeffrec}
\frac{a_{n-1}}{n(p+1)^n}-\frac{a_n}{(p+1)^n}=a_n \quad \mbox{and} \quad \frac{b_{n-1}}{np^n}+\frac{b_n}{p^n}=b_n \quad \mbox{for } n\ge 1.
\end{equation}
Now we proceed by induction on $d$. 

{\it Bases case:} For $d=0$ we indeed have $S_p(d)=1/2<1= r_p(d,d)=T_p(d)$ and hence \eqref{yeah!}.

{\it Inductive step:} Assume \eqref{yeah!} holds for all non-negative integers $d<x$, where $x$ is a positive integer. We shall prove that \eqref{yeah!} then holds for $d=x$.  

To establish the upper bound, we observe that
$$
r_p(x,x)\le \sum\limits_{0\le N\le x/p} T_p(N) \le \int\limits_{0}^{x/p} T_p(t) {\rm d}t +T_p\left(\frac{x}{p}\right),
$$
where for the first inequality we use Corollary \ref{rdd} and the induction hypothesis, and for the second inequality, we use that $T_p(t)$ is monotonically increasing for $t\ge 0$. From the second recursive formula in \eqref{abcoeffrec}, we deduce that 
$$
\int\limits_{0}^{x/p} T_p(t) {\rm d}t +T_p\left(\frac{x}{p}\right)=T_p(x),
$$
which yields the upper bound.

To establish the lower bound, we observe that
$$
r_p(x,x)\ge 1+ \sum\limits_{1\le N\le x/(p+1)} S_p(N) \ge 1+\int\limits_{0}^{x/(p+1)} S_p(t) {\rm d}t -S_p\left(\frac{x}{p+1}\right),
$$
where for the first inequality we use Corollary \ref{rdd}, the induction hypothesis and $r_p(0,0)=1$, and for the second inequality, we use that $S_p(t)$ is monotonically increasing for $t\ge 0$. From the first recursive formula in \eqref{abcoeffrec} and $a_0=1/2$, we deduce that 
$$
1+\int\limits_{0}^{x/(p+1)} S_p(t) {\rm d}t -S_p\left(\frac{x}{p+1}\right)=S_p(x),
$$
which yields the lower bound.
\end{proof}

In the following lemma, we estimate $S_p(x)$ and $T_p(x)$ by simpler functions.

\begin{lemma} \label{Taylorlemma} 
\begin{enumerate}[$($i$)$]
\item\label{taylor1} For $q\in \mathbb{N}$ and $x\ge 0$ set 
\begin{equation} \label{Taylor}
F_q(x):=\sum\limits_{n=0}^{\infty} \frac{x^n}{n! q^{n(n+1)/2}}.
\end{equation}
Then
$$
C_1 F_{p+1}(x)\le S_p(x)\le T_p(x)\le C_2F_{p}(x) \quad \mbox{for all } x\ge 1,
$$
where 
\begin{equation} \label{C1def}
C_1:=\frac{1}{2}\cdot \prod\limits_{j=1}^{\infty} \left(1-\frac{1}{(p+1)^j+1}\right)
\end{equation}
and 
\begin{equation} \label{C2def}
C_2:=\prod\limits_{j=1}^{\infty} \left(1+\frac{1}{p^j-1}\right).
\end{equation}
\item\label{taylor2} If $x\ge q\ge 2$, then we have
\begin{equation} \label{Fqx}
\frac{x^{(\log_q(x)-3)/2}}{\Gamma(\log_q(x)+1)} \le F_q(x) \le eq^{1/8}x^{(\log_q(x)-1)/2},
\end{equation}
where $\Gamma(z)$ is the usual Gamma function, defined as 
$$
\Gamma(z):=\int\limits_0^{\infty} t^{z-1}e^{-t}\dif t
$$
for $z>0$.
\end{enumerate}
\end{lemma}

\begin{proof}
\eqref{taylor1} We observe that
$$
C_1=\frac{1}{2}\cdot \lim\limits_{n\rightarrow \infty} \frac{(p+1)^{n(n+1)/2}}{\prod\limits_{j=1}^n \left((p+1)^j+1\right)}
$$
and
$$
C_2=\lim\limits_{n\rightarrow \infty} \frac{p^{n(n+1)/2}}{\prod\limits_{j=1}^n \left(p^j-1\right)}
$$
and use the definitions of $F_q(x)$, $S_p(x)$ and $T_p(x)$.
 
\eqref{taylor2} For given $x\ge q\ge 2$, the function $f\ :\ \mathbb{R}\rightarrow \mathbb{R}$ defined by
$$
f(y)=\frac{x^y}{q^{y(y+1)/2}}
$$ 
takes its maximum at $y=\log_q x -1/2$. Hence, for all $y\in \mathbb{R}$, we have
$$
f(y)\le f\left(\log_q(x)-\frac{1}{2}\right)=q^{1/8}x^{(\log_q(x)-1)/2}.
$$
The upper bound for $F_q(x)$ in \eqref{Fqx} follows now from
$$
\frac{x^n}{n! q^{n(n+1)/2}} \le \frac{q^{1/8}x^{(\log_q(x)-1)/2}}{n!}
$$ 
for all non-negative integers $n$ and
$$
\sum\limits_{n=0}^{\infty} \frac{1}{n!}=e.
$$

To get the lower bound for $F_q(x)$ we just retain one term in the Taylor series on the right-hand side of \eqref{Taylor}, namely $x^n/(n! q^{n(n+1)/2})$ with $n:=\lfloor\log_p(x)\rfloor$. It follows that
$$
F_q(x)\ge \frac{x^n}{n! q^{n(n+1)/2}}\ge \frac{x^{\log_q(x)-1}}{\Gamma\left(\log_q(x)+1\right)q^{\log_q(x)(\log_q(x)+1)/2}}=\frac{x^{(\log_q(x)-3)/2}}{\Gamma(\log_q(x)+1)},
$$
establishing the lower bound for $F_q(x)$ in \eqref{Fqx}.
\end{proof}

Combining Theorem \ref{STfunctions} and Lemma \ref{Taylorlemma}, we deduce the following.

\begin{cor} \label{rddbound} We have
$$
C_1 \cdot \frac{d^{(\log_{p+1}(d)-3)/2}}{\Gamma(\log_{p+1}(d)+1)} \le r_p(d,d) \le C_2\cdot d^{(\log_{p}(d)-1)/2} \quad \mbox{for all } d\ge 1,
$$
where $C_1$ and $C_2$ are defined as in \eqref{C1def} and \eqref{C2def}.
\end{cor}

Since $r_p(d,d)\le Z_p(d)$, the following lower bound for $Z_p(d)$ follows.

\begin{cor} \label{zlowbound} For any given prime $p$, we have
$$
Z_p(d)\ge C_1 \cdot \frac{d^{(\log_{p+1}(d)-3)/2}}{\Gamma(\log_{p+1}(d)+1)} \quad \mbox{for all } d\ge 1,
$$
where $C_1$ is defined as in \eqref{C1def}.
\end{cor}

\subsection{Upper bound for $Z_p(d)$} \label{uppbounds}
Now we bound the function $Z_p(d)$ from above. We begin by establishing a recursive upper bound.

\begin{theorem} \label{finally} We have
\begin{equation*}
Z_p(d)\le 1+\left(\log_p d+2\right)\sum\limits_{f=0}^{d-1} Z_p(f)
\end{equation*}
for all integers $d\ge 1$. 
\end{theorem}

\begin{proof} 
Let integers $M$, $d$ and $h\ge 1$ be given.  Using Proposition \ref{rhrec} and Lemma \ref{trivialobs} \eqref{trivial1} and , we obtain
\begin{equation} \label{basicrec1} 
\begin{split}
r_p^h(M,d) = & \sum\limits_{0\le e\le d} r_p^{h-1}\left(M-e\cdot \frac{p^{h+1}-1}{p-1},d-e\right)\\
\le & r_p^{h-1}(M,d)+\sum\limits_{1\le e\le d} r_p\left(M-e\cdot \frac{p^{h+1}-1}{p-1},d-e\right)\\
\le & r_p^{h-1}(M,d)+\sum\limits_{f=0}^{d-1} Z_p(f).
\end{split}
\end{equation}
Iterating this bound $g$ times, where $g\le h$, we obtain
$$
r_p^h(M,d) \le r_p^{h-g}(M,d)+g\sum\limits_{f=0}^{d-1} Z_p(f).
$$
Now, if $g= \lceil \log_p(d) \rceil+1$, then $r_p^{h-g}(M,d)=0$ because 
\begin{equation}
\begin{split}
M\ge p^h= & \frac{p^h(p-1)}{p^{h-g+1}-1}\cdot \left(p^{h-g}+p^{h-g-1}+...+1\right)\\
 > & p^{g-1}(p-1)\left(p^{h-g}+p^{h-g-1}+...+1\right)\\
\ge & d\left(p^{h-g}+p^{h-g-1}+...+1\right).
\end{split}
\end{equation}
Thus, if $\lceil \log_p(d) \rceil+1\le h$, then
$$
r_p^h(M,d)\le \left(\log_p d+2\right)\sum\limits_{f=0}^{d-1} Z_p(f).
$$ 
Otherwise, by iterating \eqref{basicrec1} $h$ times, we get
\begin{equation*} 
r_p^h(M,d) \le r_p^0(M,d)+h\sum\limits_{f=0}^{d-1} Z_p(f)\le 1+\left(\log_p d+2\right)\sum\limits_{f=0}^{d-1} Z_p(f).
\end{equation*}
So in any case,
$$
r_p^h(M,d)\le 1+\left(\log_p d+2\right)\sum\limits_{f=0}^{d-1} Z_p(f).
$$
Now, if $h=\lfloor \log_p(M) \rfloor$, then $r_p^h(M,d)=r_p(M,d)$. It follows that
$$
r_p(M,d)\le 1+\left(\log_p d+2\right)\sum\limits_{f=0}^{d-1} Z_p(f).
$$
The claim follows upon taking the maximum over all integers $M$.
\end{proof}

Now we are ready to establish the following explicit upper bound for $Z_p(d)$.

\begin{theorem} \label{Zpdupp} For any given prime $p$, we have
\begin{equation} \label{Zpdbound}
Z_p(d)\le \left(\log_p(d+1)+3\right)^d\quad \mbox{for all integers } d\ge 0.
\end{equation}
\end{theorem}

\begin{proof}
We prove the claim by induction over $d$. 

{\it Base case:} $d=0$. Then $Z_p(d)=1$, and hence \eqref{Zpdbound} holds.

{\it Inductive step:} Assume \eqref{Zpdbound} holds for all non-negative integers $d<g$, where $g\ge 1$ is an integer. We prove that this bound then holds for $d=g$.  By Theorem \ref{finally} and the induction hypothesis, we have
\begin{equation}
\begin{split}
Z_p(g)\le & 1+\left(\log_p g+2\right)\sum\limits_{f=0}^{g-1} \left(\log_p(f+1)+3\right)^f\\ 
\le & 1+\left(\log_p g+2\right)\sum\limits_{f=0}^{g-1} \left(\log_p g+3\right)^f\\
= & 1+\left(\log_p g+2\right)\cdot \frac{\left(\log_p g+3\right)^g-1}{\log_p g+2}\\
\le & \left(\log_p(g+1)+3\right)^g.
\end{split}
\end{equation}
\end{proof}

\subsection{Remarks on $r_p(M,d)$} 
We introduce the following notation.

\begin{definition} \label{SpD}
For a non-negative integer $D$ denote by $\sigma_p(D)$ the sum of digits in its $p$-adic presentation.
\end{definition}

Given an integer $d$, it would be desirable to have a lower bound for $r_p(M,d)$, as $M$ runs through the integers. However, it is not possible to establish a simple non-trivial lower bound of this kind because the set of non-negative integers $M$ for which $r_p(M,d)=0$ is infinite by Lemma \ref{rqrelation} and the following observation.

\begin{lemma} \label{SpDobs} If $\sigma_p(D)>d$, then $q_p(D,d)=0$. 
\end{lemma}

\begin{proof}
If
$$
n_0+n_1p+n_2p^2+...=D
$$
with $n_0,n_1,n_2,...$ non-negative integers, then $n_0+n_1+...\ge \sigma_p(D)$. Hence, using Definition \ref{qQdef}, it follows that $q_p(D,d)=0$ if $\sigma_p(D)>d$.
\end{proof}

\begin{remark}\label{future}{\rm
It would be of independent number theoretic interest to obtain more detailed information on the behaviour of the partition function $r_p(M,d)$, as $d$ is fixed. It turns out that this function behaves highly irregularly. The following graph displays $r_p(M,d)$ for $p=3$, $d=10$ and $0\le M\le 10000$. 

\includegraphics[width=\textwidth]{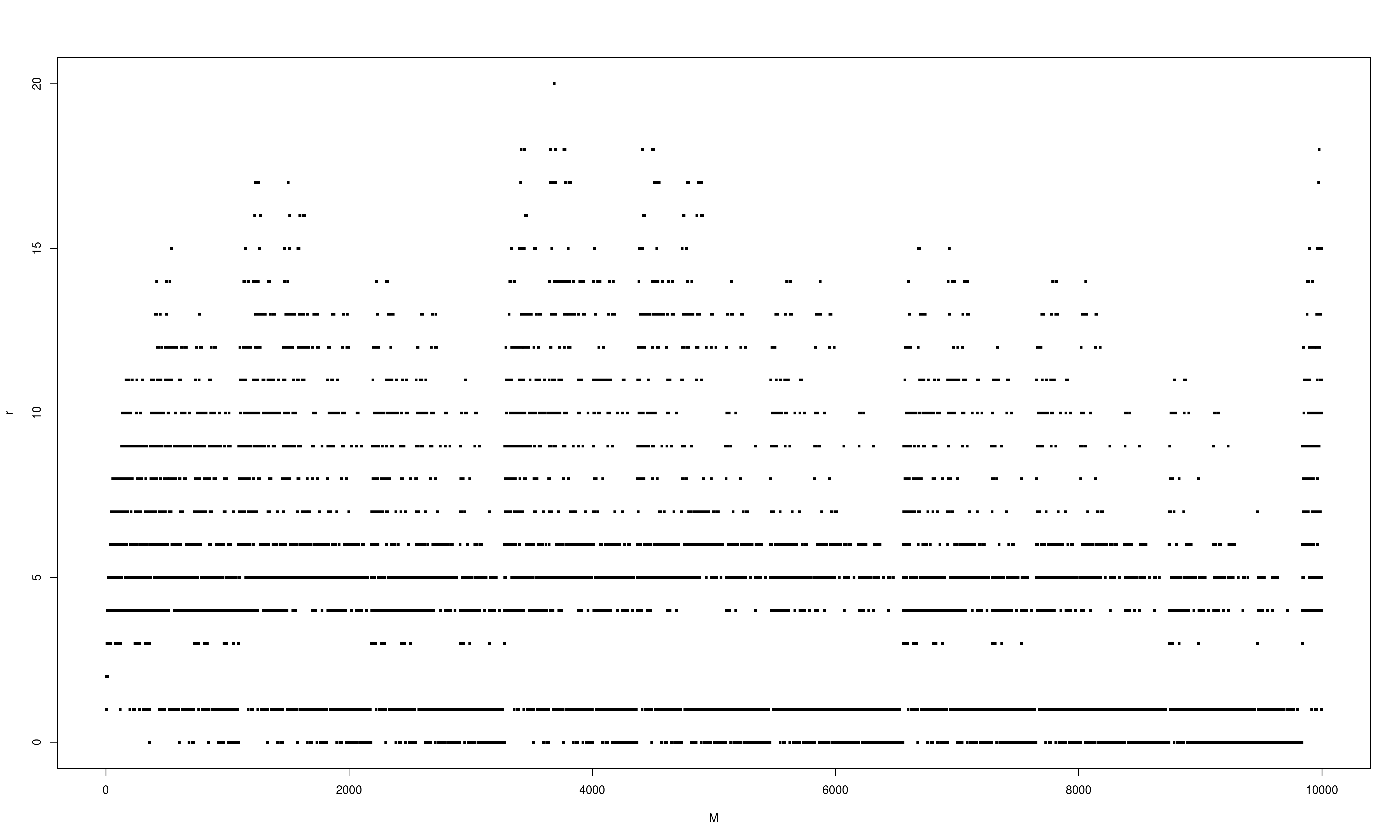}

 In the following section, we give upper and lower bounds for the maximum dimension of $\Ext^k(\Delta_m,\Delta_\ell)$ where $\ell$ runs and $k$ and $m$ are fixed. We expect significant improvements on these bounds through ongoing research by the first-named author on the fine structure of $r_p(M,d)$.} 
\end{remark}

\section{Estimates for the dimension}\label{secestimates}
We are interested in the functions below.

\begin{definition} For every non-negative integer $k$ and positive integer $m$ let
$$
X_m(k):=\max\limits_{\ell\in \mathbb{N}} \dim \Ext^k(\Delta_m,\Delta_\ell).
$$
Further, define
$$
X(k):=\max\limits_{m\in \mathbb{N}} X_m(k)= \max\limits_{m,\ell\in \mathbb{N}} \dim \Ext^k(\Delta_m,\Delta_\ell).
$$
\end{definition}

We now use our results on $r_p(d,d)$ and $Z_p(d)$ to derive a lower bound for $X_m(k)$ and an upper bound for $X(k)$. In particular, we shall see that $X(k)$ is finite, i.e. $\dim \Ext^k(\Delta_m,\Delta_\ell)$ stays bounded  as $k$ is fixed and $m$ and $\ell$ run through the positive integers.

\subsection{Lower bound for the dimension}
We prove the following lower bound for $X_m(k)$.

\begin{theorem}\label{lowerthm} For all integers $k\ge 10$ and $m\ge 1$ we have
\begin{equation} \label{thelowerbound}
X_m(k)\ge C_1\cdot \frac{k^{\log_{p+1}(k)/2-6}}{\Gamma(\log_{p+1}(k)+1)},
\end{equation}
where $C_1$ is defined as in \eqref{C1def}.
\end{theorem}

\begin{remark}{\rm Note that for $m=1$, there is a stronger result in \cite{EHP}, proving that $X_m(k)$ grows at least exponentially. We hope to improve our result to show that this is true for arbitrary $m$ through future work referred to in Remark \ref{future}.}\end{remark}

\begin{proof}
Using Proposition \ref{dimtheorem} and Lemma \ref{A1B}, we have
\begin{equation} \label{dimBqk}
\dim \Ext^k(\Delta_m,\Delta_\ell)\ge B^{w_1,...,w_q}_{q,k},
\end{equation}
where 
$$
\ell-m=w_1p^{q-1}+w_2p^{q-2}+...+w_q
$$
upon recalling Definition \ref{longdef} and \eqref{wg}.  Further, from Theorem \ref{expB} and Lemma \ref{rqrelation}, we deduce that
\begin{equation} \label{so}
\begin{split}
B&^{w_1,...,w_q}_{q,k} \\& = \delta(\ell-m-k;2(p-1))
\sum\limits_{0\le d\le k/2} \ \sum\limits_{\substack{\frac{\ell-m-p(k-2d)}{2p^2}\le D\le \frac{\ell-m-(k-2d)}{2p^2}\\ D\equiv d \bmod{p-1}}}
r_p\left(\frac{D-d}{p-1},d\right)
\end{split}
\end{equation}
for $\ell\ge m$.

We claim that there exist $\ell,d,D$ satisfying the following conditions:
\begin{eqnarray}
\ell & \ge & m\\
\ell-m&\equiv& k\bmod{2(p-1)}\\
\frac{k}{6} \le & d & \le \frac{k}{2}\\
\frac{\ell-m-p(k-2d)}{2p^2}\le & D &\le \frac{\ell-m-(k-2d)}{2p^2}\\
D &\equiv& d \bmod{(p-1)}\\
\frac{d}{2}\le & \frac{D-d}{p-1} & \le d.
\end{eqnarray}
If this is the case, then it follows from \eqref{so}, Lemma \ref{trivialobs}\eqref{trivial2}, Corollary \ref{rddbound} and $k/12\le (D-d)/(p-1)\le k$  that
\begin{equation}
\begin{split}
B^{w_1,...,w_q}_{q,k}&\ge  r_p\left(\frac{D-d}{p-1},d\right) =r_p\left(\frac{D-d}{p-1},\frac{D-d}{p-1}\right) \\&\ge C_1 \cdot \frac{(k/12)^{(\log_{p+1}(k/12)-3)/2}}{\Gamma(\log_{p+1}(k)+1)} \\
&\ge  C_1\cdot \frac{k^{\log_{p+1}(k)/2-5}}{\Gamma(\log_{p+1}(k)+1)}
\end{split}
\end{equation}
which together with \eqref{dimBqk} establishes \eqref{thelowerbound}.

It remains to verify the above claim. Indeed, if $p\ge 3$, then
$$
d=\left[\frac{k}{5}\right], \quad D=p\cdot \left[\frac{k}{5}\right], \quad \ell=m+2p\left(p^2-1\right)\cdot \left[\frac{k}{5}\right]+2(p-1)\cdot \left[\frac{k}{2}\right]
$$
does the job, and if $p=2$, then we take
$$
d=\left[\frac{k}{5}\right], \quad D=2\left[\frac{k}{5}\right], \quad \ell=m+16\left[\frac{k}{5}\right]+k.
$$
\end{proof}

\subsection{Upper bound for the dimension}
We want to establish an upper bound for $A_{h,k}^{w_1,...,w_h}$ that only depends on $k$. To this end, we prove the following combinatorial result involving the function defined in Definition \ref{SpD}.

\begin{lemma} \label{combi} Let $D$ and $r$ be non-negative integers. By $N_p(D,r)$ denote the number of $g$-tuples $(u_1,...,u_g)$ such that $u_i\in \{0,1\}$ for $i=1,...,g$,
\begin{equation} \label{Drcondition1}
D\ge u_1p^{g-1}+u_2p^{g-2}+...+u_g, 
\end{equation}
\begin{equation} \label{Drcondition2}
u_1+...+u_g\le r
\end{equation}
and 
$$
\sigma_p(D-(u_1p^{g-1}+u_2p^{g-2}+...+u_g))\le r/2-(u_1+....+u_g)/2.
$$
Then
\begin{equation} \label{dasistderclaim}
N_p(D,r)\le 32^{r}.
\end{equation}
\end{lemma}

\begin{proof}  We first observe that 
\begin{equation*} \label{firststep}
\sigma_p(D)> r \Longrightarrow N_p(D,r)=0
\end{equation*}
because if $\sigma_p(D)>r$, then 
\begin{equation*}
\begin{split}
 \sigma_p(D-(u_1p^{g-1}+u_2p^{g-2}+...+u_g))&\ge \sigma_p(D)-(u_1+...+u_g)\\
  \ge r-(u_1+...+u_g)&\ge r/2-(u_1+....+u_g)/2
\end{split}
\end{equation*}
whenever $g\in \mathbb{N}$, $u_1,...,u_g\in \{0,1\}$ and the conditions in \eqref{Drcondition1} and \eqref{Drcondition2} are satisfied.

So we may suppose $\sigma_p(D)\le r$ throughout the remainder of this proof. Then at most $r$ of the digits of $D$ are non-zero. Suppose that $a_1a_2...a_g$ is the $p$-adic presentation of $D$, and $\mathcal{I}$ is the set of indices $i$ for which $a_i\not=0$. Hence, $\sharp \mathcal{I}\le r$. 
Now let 
$$
\mathcal{D}:=\left\{D-\sum\limits_{i\in \mathcal{I}} u_ip^{h-i} \ :\ u_i\in \{0,1\} \mbox{ for } i\in \mathcal{I}\right\}.
$$
We observe that
$$
\sharp \mathcal{D} = 2^{\sharp\mathcal{I}}\le 2^r. 
$$

Further, let $n\in \mathcal{D}$. Suppose that $b_1b_2...b_g$ is the $p$-adic presentation of $n$ (with $g$ digits, as that of $D$), and let $\mathcal{J}_n$ be the set of indices $j$ for which $b_j\not=0$. Clearly, $\mathcal{J}_n\subseteq \mathcal{I}$ and thus $\sharp \mathcal{J}_n \le r$.  Set
\begin{equation*}
\begin{split}
\mathcal{N}_n := \Big\{ & n-\sum\limits_{\substack{1\le i \le g \\ i\not\in \mathcal{J}_n}} u_ip^{h-i}\ge 0\ : \ u_i\in \{0,1\} \mbox{ if } 1\le i\le g \mbox{ and } i\not\in \mathcal{J}_n,\\ &
\mbox{ and the inequalities \eqref{e4} and \eqref{c5} below are satisfied} \Big\},
\end{split}
\end{equation*}
where the said inequalities are 
\begin{equation} \label{e4}
\sigma_p(n-\sum\limits_{\substack{1\le i \le g \\ i\not\in \mathcal{J}_n}} u_ip^{h-i})\le r
\end{equation}
and 
\begin{equation} \label{c5}
\sum\limits_{\substack{1\le i \le g \\ i\not\in \mathcal{J}_n}} u_i\le r.
\end{equation}
We aim to prove that
\begin{equation} \label{daswollenwir}
\sharp\mathcal{N}_n\le 16^{r}.
\end{equation}
Then it follows that
$$
N_p(D,r)\le \sum\limits_{n\in \mathcal{D}} \sharp \mathcal{N}_n \le \sharp\mathcal{D}\cdot 16^r\le 32^{r},
$$
and hence we get \eqref{dasistderclaim} and thus the statement of the lemma.

To prove \eqref{daswollenwir}, we bound the sum of $p$-adic digits of a non-negative number $N\in \mathcal{N}_n$ of the form
$$
N=n-\sum\limits_{\substack{1\le j\le g\\ j\not\in \mathcal{J}_n}} u_jp^{g-j} \quad \mbox{with } u_j\in \{0,1\} \mbox{ if } 1\le j\le g \mbox{ and } j\not\in \mathcal{J}_n 
$$
from below.  Let $\mathcal{J}_n=\{j_1,...,j_t\}$ with $j_1<j_2<...<j_t$. For $m=1,...,t$, set
$$
l_m:=\begin{cases} \max\limits_{\substack{j_m<l<j_{m+1}\\ u_l=1}} l & \mbox{ if } u_l=1 \mbox{ for some } j_m<l<j_{m+1}, \\ j_m & \mbox{ otherwise,}\end{cases}
$$
where $j_{t+1}:=g+1$, and 
$$
\delta_m:=\begin{cases} 1 & \mbox{ if } l_m>j_m, \\ 0 & \mbox{ otherwise.}\end{cases}
$$
Let $c_1c_2...c_g$ be the $p$-adic representation of $N$. 
Then we have
\begin{equation*}
\begin{split}
& c_1=0,...,c_{j_1-1}=0,\\
& c_{j_1}=b_{j_1}-\delta_1,c_{j_1+1}=p-1-u_{j_1+1},...,c_{l_1-1}=p-1-u_{l_1-1},c_{l_1}=p-1,\\ & c_{l_1+1}=0,...,c_{j_2-1}=0,\\
& c_{j_2}=b_{j_2}-\delta_2,c_{j_2+1}=p-1-u_{j_2+1},...,c_{l_2-1}=p-1-u_{l_2-1},c_{l_2}=p-1,\\ & c_{l_2+1}=0,...,c_{j_3-1}=0,\\
& \cdots,\\
& c_{j_t}=b_{j_t}-\delta_t,c_{j_t+1}=p-1-u_{j_t+1},...,c_{l_t-1}=p-1-u_{l_t-1},c_{l_t}=p-1,\\ & c_{l_t+1}=0,...,c_{j_{t+1}-1}=0,
\end{split}
\end{equation*}
where we interpret a chain of digits 
$$
c_{j_m+1}=p-1-u_{j_m+1},...,c_{l_m-1}=p-1-u_{l_m-1},c_{l_m}=p-1
$$
as empty if $l_m=j_m$. It follows that 
\begin{equation} \label{Nf3}
\sigma_p(N)\ge \sum\limits_{m=1}^t(l_m-j_m)-\sum\limits_{\substack{1\le i \le g \\ i\not\in \mathcal{J}_n}} u_i.
\end{equation}

Now, \eqref{e4}, \eqref{c5} and \eqref{Nf3} imply that
$$
\sum\limits_{m=1}^t (l_m-j_m)\le 2r.
$$
Recalling that $t\le r$, we deduce that
\begin{equation*}
\begin{split}
\sharp \mathcal{N}_n \le & 2^r \cdot \sharp \Big\{ (l_1,...,l_t)\in \mathbb{N}^t \ : \ \sum\limits_{m=1}^{t} (l_m-j_m)\le 2r, \\ & j_1\le l_1<j_2\le l_2<...<j_t\le l_t<j_{t+1}\Big\}\\
\le & 2^r\cdot \sharp\{(n_1,...,n_r)\in \mathbb{Z}^r_{\ge 0} \ : n_1+...+n_r\le 2r\} \\
= & 2^r\cdot \sharp\{(n_1,...,n_{r+1})\in \mathbb{Z}^{r+1}_{\ge 0} \ : n_1+...+n_{r+1}=2r\} \\
= & 2^r\cdot \binom{3r}{r}\le 2^r\cdot 2^{3r}=16^r, 
\end{split}
\end{equation*}
where the factor $2^r$ comes from the possible choices of the $u_g's$, and we use the fact that
$$
\sharp\{(n_1,...,n_r,n_{r+1})\in \mathbb{Z}^{r+1}_{\ge 0} \ : n_1+...+n_r+n_{r+1}=s\}=\binom{r+s}{r}
$$
equals the number of choices of $r$ elements from a set of $s+1$ elements with possible repitition. This establishes \eqref{daswollenwir} and thus completes the proof. 
\end{proof}

Using Lemma \ref{combi}, we now deduce the following.

\begin{proposition} \label{merest} Let $h$ and $k$ be integers such that $h\ge 1$ and $k\ge 0$. Then 
\begin{equation} \label{schoen}
A^{w_1,...,w_h}_{h,k}\le (k+1)^2\left(32\log_p(k+1)+96\right)^k.
\end{equation}
\end{proposition}

\begin{proof} Write
$$
W:=w_1p^{h-1}+w_2p^{h-2}+...+w_h.
$$
Hence, Lemma \ref{A1B}, Remark \ref{rem2} and Theorem \ref{expB} give
\begin{equation} \label{Aabsch}
\begin{split}
A^{w_1,...,w_h}_{h,k} \le\sum\limits_{(u_1,...,u_h)\in \{0,1\}^h} \ \sum\limits_{0\le d\le \frac{k-u}{2}} \sum\limits_{\frac{W-2U-pk}{2p^2}\le D\le \frac{W-2U}{2p^2}} q_p(D,d),
\end{split}
\end{equation}
where we set
$$
u:=u_1+...+u_h \quad \mbox{and} \quad U:=u_1p^{h-1}+u_2p^{h-2}+...+u_h.
$$
If $h\le 2$, then \eqref{Aabsch} together with Lemma \ref{rqrelation} and Theorem \ref{Zpdupp} imply \eqref{schoen}. Assume now that $h>2$. Then from \eqref{Aabsch}, we further deduce 
\begin{equation} \label{ditte}
\begin{split}
A^{w_1,...,w_h}_{h,k}& \le  4\sum\limits_{(u_1,...,u_{h-2})\in \{0,1\}^{h-2}} \ \sum\limits_{0\le d\le \frac{k-\tilde{u}}{2}} \
\sum\limits_{\frac{W-(pk+2p+2)}{2p^2}-\tilde{U}\le D\le \frac{W}{2p^2}-\tilde{U}} q_p(D,d) \\ &
\le 4 \sum\limits_{\frac{W-(pk+2p+2)}{2p^2}\le \tilde{D}\le \frac{W}{2p^2}} \ \sum\limits_{(u_1,...,u_{h-2})\in \{0,1\}^{h-2}} \ \sum\limits_{0\le d\le\frac{k-\tilde{u}}{2}} q_p(\tilde{D}-\tilde{U},d),
\end{split}
\end{equation}
where we set
$$
\tilde{u}:=u_1+...+u_{h-2} \quad \mbox{and} \quad \tilde{U}:=u_1p^{h-3}+u_2p^{h-4}+...+u_{h-2}.
$$
Using Lemma \ref{SpDobs} and Lemma \ref{combi}, we observe that for any given $\tilde{D}$, the set $\mathcal{S}(\tilde{D})$ of tuples $(u_1,...,u_{h-2})\in \{0,1\}^{h-2}$ such that
$$
\sum\limits_{0\le d\le k/2-\tilde{u}/2} q_p(\tilde{D}-\tilde{U},d) >0
$$
has cardinality at most
\begin{equation} \label{card}
\sharp \mathcal{S}(D) \le 32^{k}.
\end{equation}
Now Theorem \ref{Zpdupp}, \eqref{ditte} and \eqref{card} give the desired bound \eqref{schoen}. 
\end{proof}

Now we are ready to prove the following upper bound for $X(k)$. 

\begin{theorem}\label{upperthm}
Let $k$ be a non-negative integer. Then we have 
\begin{equation} \label{theupperbound}
X(k)\le (k+4)^3\left(32\log_p(k+1)+96\right)^k.
\end{equation}
\end{theorem}

\begin{proof} The number of summands in \eqref{dimform2} is at most 1 because $t_g=p+1-s_g$ implies that $w_g=p+1-2s_g$ is even and hence $w_g\not=1$ if $p>2$, and $w_g=p+1-2s_g$ is odd and hence $W_g\not\equiv 0 \bmod{2}$ if $p=2$. 

It remains to bound the number of non-zero summands in \eqref{dimform4}. Assume that $f$ is the smallest and $F$ is the largest natural number such that $1\le f<F\le q$, $W_f\equiv 1 \bmod{2}$,
$$
t_{f+1}=p+1-s_{f+1} \mbox{ and } s_{f+1}\not=p,..., t_{F}=p+1-s_{F} \mbox{ and } s_{F}\not=p
$$ 
and
$$
t_{F+1}=p+1-s_{F+1},...,t_q=p+1-s_q,
$$
provided they exist.
Then the sum on the right-hand side of \eqref{dimform4} runs precisely from $h=f$ to $h=F-1$. However, if $F-1\ge h>f+k$, then $A^{w_1,...,w_h}_{h,k}=0$ by Definition \ref{Adef1} because necessarily $w_g\ge 3-p$ and $c_g\equiv W_{g-1}\equiv 1 \bmod{2}$ and hence $c_g\ge 1$ for $f+1\le g\le h$, which implies
$$
(p-1)c_g+w_g-u_g\ge (p-1)+(3-p)-1=1\quad  \mbox{for } f+1\le g\le h.
$$
Therefore, the number of non-zero summands in \eqref{dimform4} is restricted by $k+1$. From these observations, \eqref{dimformel} and Proposition \ref{merest}, we deduce the result. 
\end{proof}

\section{An algorithm to calculate the dimension}\label{secalg}
In this section, we formulate an algorithm to calculate the precise value of the dimension. We begin by ruling out cases in which the dimension is 0.

\begin{lemma} \label{rulingout}
If not $0\le k\le \ell-m$, then $\dim \Ext^k(\Delta_m,\Delta_\ell)=0$. 
\end{lemma}

\begin{proof}
First note that $\dim \Ext^k(\Delta_m,\Delta_\ell)=0$ if $\ell-m<0$ by standard theory of quasi-hereditary algebras. We can also see this combinatorially, recalling that
$$
\ell-m=w_1p^{q-1}+w_2p^{q-2}+...+w_q.
$$
If now $e-m<0$, then necessarily $w_1=w_2=...=w_f=0$ and $w_{f+1}<0$ for some $f$ with $0\le f\le q-1$ because otherwise $w_1=...=w_q=0$, and hence $\ell-m=0$, or $w_1=...=w_g=0$ and $w_{g+1}>0$ for some $g$ with $0\le g\le q-1$, which implies that
\begin{equation*}
\begin{split}
\ell-m= & w_{g+1}p^{q-g-1}+w_{g+2}p^{q-g-2}+...+w_q\\
\ge & p^{q-g-1}-\left((p-1)p^{q-g-2}+(p-1)p^{q-g-3}+...+(p-1)\right)=1.
\end{split}
\end{equation*}
Using Definition \ref{Adef1}, it follows that
$$
A^{w_1,...,w_h}_{h,k}=0 \quad \mbox{if } h\ge f+1.
$$
Looking at the definitions of $D_1$, $D_2$, $D_3$, $D_4$ in \eqref{dimform1}, \eqref{dimform2}, \eqref{dimform3}, \eqref{dimform4}, we deduce that $D_1=D_2=D_3=D_4=0$
which implies $\dim \Ext^k(\Delta_m,\Delta_\ell)=0$ by \eqref{dimformel}. 

Now let $\ell-m\ge 0$. For $1\le h\le q$ set 
$$
V_h:=w_1p^{h-1}+w_2p^{h-2}+...+w_{h}.
$$
From Lemma {\rm \ref{A1B}}, Remark {\rm \ref{rem2}} and Remark {\rm \ref{rem3}}, it follows that $A^{w_1,...,w_h}_{h,k}=0$ if not $0\le k\le V_h$. Moreover,  $V_h\le \ell-m$ because if 
$V_h\le 0$, then trivially $V_h\le \ell-m$, and if $V_h\ge 1$, then
\begin{equation*}
\begin{split}
V_h\le & p^{q-h}V_h-\left(p^{q-h}-1\right)\\ 
= & p^{q-h}V_h-\left((p-1)p^{q-h-1}+(p-1)p^{q-h-2}+...+(p-1)\right)\\
\le & w_1p^{q-1}+w_2p^{q-2}+...+w_q=\ell-m. 
\end{split}
\end{equation*}
Therefore, $A^{w_1,...,w_h}_{h,k}=0$ for all $h\in \{1,...,q\}$ if not $0\le k\le \ell-m$.  Hence, in this case we have $D_1=D_2=D_4=0$, and $D_3=0$ holds trivially. By \eqref{dimformel}, this implies that $\dim \Ext^k(\Delta_m,\Delta_\ell)=0$ if not $0\le k\le \ell-m$, which completes the proof.
\end{proof}

Now we are ready to formulate an\\ \\
{\bf Algorithm to calculate $\dim \,\Ext^k(\Delta_m,\Delta_\ell)$:}
\begin{itemize}
\item Input: $p$, $k$, $m$, $\ell$.
\item Compute the digits $s_i-1$ of $m-1$ and the digits $t_i-1$ of $\ell-1$ in their $p$-adic presentations.
\item If not $0\le k\le \ell-m$, then  $\dim \,\Ext^k(\Delta_m,\Delta_\ell)=0$. Otherwise, proceed below.
\item Use Proposition \ref{dimtheorem} to relate $\dim \,\Ext^k(\Delta_m,\Delta_\ell)=0$ to $A_{h,k}^{w_1,...,w_h}$.
\item Compute $A_{h,k}^{w_1,...,w_h}$ recursively using Corollary \ref{recu2A}. 
\end{itemize}

In the next section, this algorithm is implemented as a computer program in C.

\section{A C program to calculate the dimension}\label{program}
The first program computes the dimension $\dim \,\Ext^k(\Delta_m,\Delta_\ell)$. Input variables are $p$, $k$, $m$ and $e:=\ell$. The second program spits out a table of the dimensions for given $p$, $k$ and $1\le m,e\le p^q$. Input variables are $p$, $q$ and $k$.
\lstset{ %
	language=C,                
	basicstyle=\footnotesize,       
	numbers=left,                   
	numberstyle=\footnotesize,      
	stepnumber=1,                   
	numbersep=5pt,                  
	backgroundcolor=\color{white},  
	showspaces=false,               
	showstringspaces=false,         
	showtabs=false,                 
	frame=single,	                
	tabsize=2,	                
	captionpos=b,                   
	breaklines=true,                
	breakatwhitespace=false,        
	morecomment=[l][keywordstyle]{\#},
	mathescape=true,
	escapechar=\@
	}
	\lstset{keywordstyle=\color{NavyBlue}\bfseries\emph,
	commentstyle=\color{YellowOrange}\bfseries,
	stringstyle=\color{Green}}
	\lstset{emph={getdigits,dimension,A2,A,fprintf,exit,info,main,testSet,scanf,pow,pthread_create,printf,time,atoi,pthread_cancel,omp_get_max_threads,sleep,fflush,malloc,recursiveSets}, emphstyle=\color{blue}\bfseries}
	\lstset{emph={[2]unsigned,int,long,uint64_t,char,const,void,pthread_t,size_t,time_t,intmax_t}, emphstyle={[2]\color{BrickRed}\bfseries}}
	
\begin{lstlisting}[title={dimension.c}]
#include <stdlib.h>
#include <stdio.h>
#include <string.h>
#include <math.h>
#include <stdint.h>


// See Corollary @\ref{recu2A}@
intmax_t A(intmax_t p, intmax_t h, intmax_t k, intmax_t * w) {
	
	if (h == 1) {
		int u;
		for (u = 0; u < 2; u++) {
			if ((k+ u - w[1]) % (2*(p-1)) == 0
				&& w[1] + (2 * (p-1) * u) <= (k + u) * p
				&& (k + u) * p <= w[1] * p
			)
			return 1;
		}
		return 0;
	} else {
		int u;
		intmax_t d;
		intmax_t result = 0;
		for (u = 0; u < 2; u++) {
			for (d = 0; d <= ((w[1]-2*u)/(2*p)); d++) {
				intmax_t w_copy[h+1];
				memcpy(w_copy, w, sizeof(w_copy));
				w_copy[2] += p * (w[1] - (2*u) - (2*d*p));
				result += A(p, h-1, k - u - 2*d, w_copy+1);
			}
		}
		return result;
	}
	
}

void getdigits(intmax_t *s, intmax_t p, intmax_t q, intmax_t m) { 
	size_t i;
	
	m--;
	for (i = 0; i < q; i++) {
		
		intmax_t digit = m % p;
		
		s[q - i] = digit + 1;
		m-= digit;
		m /= p;
	}
}

intmax_t compute_q (intmax_t p, intmax_t m, intmax_t e) {
	intmax_t l1 = ceil(log(m)/log(p));
	intmax_t l2 = ceil(log(e)/log(p));
	
	return (l1 > l2 ? l1 : l2);
}

// See  Theorem @\ref{dimtheorem}@
intmax_t dimension( intmax_t p, intmax_t q, intmax_t k, intmax_t m, intmax_t e, int stopAqk) {
	if (k < 0 || k > (e-m))
		return 0;
	
	size_t i;
	
	if (!q)
		q = compute_q (p, m, e);
	
	intmax_t s[128];
	intmax_t t[128];
	
	getdigits(s, p, q, m);
	getdigits(t, p, q, e);
	
	intmax_t w[q+1]; // See @\eqref{wg}@
	
	for (i = 1; i <= q; i++) {
		w[i] = t[i] - s[i];
	}
	
	intmax_t W[q+1]; // See @\eqref{Wf}@
	W[0] = 0;
	
	for (i = 1; i <= q; i++) {
		if (p >= 3)
			W[i] = W[i-1] - w[i];
		else if (p == 2)
			W[i] = w[i];
		else {
			fprintf(stderr, "Undefined for p < 2\n");
			exit(-1);
		}
	}
	
	
	intmax_t D = A(p, q,k,w); // term in @\eqref{dimform1}@
	
	if (stopAqk)
		return D;
		
	intmax_t h;
	
	for (h = 1; h < q; h++) { // term in @\eqref{dimform2}@
		size_t l, E = 1;
		for (l = h+2; l <= q; l++) {
			if (t[l] != (p+1- s[l]))
				E = 0;
		}
		if (E != 0 && (W[h] % 2 == 0) && (w[h+1] == 1))
			D += A(p, h, k, w);
	}
	
	for (h = 1; h < q; h++) { // term in @\eqref{dimform4}@
		size_t l, E = 1;
		for (l = h+1; l <= q; l++) {
			if (t[l] != (p+1- s[l]))
				E = 0;
		}
		if (E != 0 && (W[h] & 1) && (s[h+1] != p)) {
			intmax_t result = A(p, h, k, w);
			D += result;
		}
	}
	
	size_t l, E = 1; // term in @\eqref{dimform3}@
	for (l = 2; l <= q; l++) {
		if (t[l] != (p+1- s[l]))
			E = 0;
	}
	if (E != 0 && k == 0 && w[1] == 1)
		D++;
	
	return D;
}


#ifndef LIBDIM

int main(int argc, char** argv) {
	
	intmax_t p = -1;
	intmax_t q = -1;
	intmax_t k = -1;
	intmax_t m = -1;
	intmax_t e = -1;
	
	fprintf(stderr, "Warning: p has to be odd prime >= 2. This is not checked\n");
	
	
	if (argc > 1) {
		argv++;
		while (*argv) {
			switch (**argv) {
				case '-':
				switch ((*argv)[1]) {
					case 'p': p = atoi(*(++argv)); break;
					case 'q': q = atoi(*(++argv)); break;
					case 'k': k = atoi(*(++argv)); break;
					case 'm': m = atoi(*(++argv)); break;
					case 'e': e = atoi(*(++argv)); break;
					default:
					fprintf(stderr, "Unsupported argument %s\n", *argv);
					exit(-1);
				}
				break;
				default:
				fprintf(stderr, "Unsupported argument %s\n", *argv);
				exit(-1);
			}
			argv++;
		}
		
	}
	
	if (p < 2) {
		printf("input p: ");
		fflush(stdout);
		scanf("%ld", &p);
	}
	
	while (k < 0) {
		printf("input k: ");
		fflush(stdout);
		scanf("%ld", &k);
	}
	
	if (m < 1) {
		printf("input m: ");
		fflush(stdout);
		scanf("%ld", &m);
	}
	
	if (e < 1) {
		printf("input e: ");
		fflush(stdout);
		scanf("%ld", &e);
	}
	
	intmax_t D = dimension(p, 0, k, m, e, 0);
	printf("The dimension equals %ld\n", D);
	
	return 0;
}

#endif
\end{lstlisting}
\begin{lstlisting}[title={dimension\_search.c}]
#include <stdlib.h>
#include <stdio.h>
#include <string.h>
#include <stdint.h>
#include <math.h>

intmax_t dimension( intmax_t p, intmax_t q, intmax_t k, intmax_t m, intmax_t e, int);


int main(int argc, char** argv) {
	
	intmax_t p = -1;
	intmax_t q = -1;
	intmax_t k = -1;
	intmax_t m = -1;
	intmax_t e = -1;
	
	int print_only_Aqk = 0;
	
	fprintf(stderr, "Warning: p has to be odd prime >= 2. This is not checked\n");
	
	
	if (argc > 1) {
		argv++;
		while (*argv) {
			switch (**argv) {
				case '-':
				switch ((*argv)[1]) {
					case 'p': p = atoi(*(++argv)); break;
					case 'q': q = atoi(*(++argv)); break;
					case 'k': k = atoi(*(++argv)); break;
					case 'm': m = atoi(*(++argv)); break;
					case 'e': e = atoi(*(++argv)); break;
					case 'A': print_only_Aqk = 1; break;
					default:
					fprintf(stderr, "Unsupported argument %s\n", *argv);
					exit(-1);
				}
				break;
				default:
				fprintf(stderr, "Unsupported argument %s\n", *argv);
				exit(-1);
			}
			argv++;
		}
		
	}
	
	if (p < 2) {
		printf("input p: ");
		fflush(stdout);
		scanf("%ld", &p);
	}
	
	while (q <= 1) {
		printf("input q: ");
		fflush(stdout);
		scanf("%ld", &q);
	}
	
	while (k < 0) {
		printf("input k: ");
		fflush(stdout);
		scanf("%ld", &k);
	}
	
	if (k < 0 || k >= pow(p, q)) {
		printf("dimension = 0\n");
		return 0;
	}
	
	if (print_only_Aqk)
		printf("\"A(q,k)\"");
	else
	printf("D");
	
	if (m > 0) {
		printf("\ne\n");
		for (e = 1; e <= pow(p,q); e++) {
			intmax_t D;
			D = dimension(p, q, k, m, e, print_only_Aqk);
			printf("%lu,%lu\n", e, D);
		}
		
	}
	else if (e > 0) {
		printf("\nm\n");
		for (m = 1; m <= pow(p,q); m++) {
			intmax_t D;
			D = dimension(p, q, k, m, e, print_only_Aqk);
			printf("%lu,%lu\n", m, D);
		}
	} else {
		printf(",m");

		printf("\ne\n");
		for (e = 1; e <= pow(p,q); e++) {
			printf("%lu,", e);
			for (m = 1; m <= pow(p,q); m++) {
				intmax_t D;
				D = dimension(p, q, k, m, e, print_only_Aqk);
				printf(",%lu", D);
			}
			printf("\n");
		}
	}
	
	return 0;
}




\end{lstlisting}
	

\section{Program to calculate $e$ and $m$ from $\lambda$ and $\mu$}\label{lambdamu}

In this section, for the convenience of the reader, we implement an algorithm how to compute the numbers $m$ and $e:=\ell$ of two standard modules in a block of polynomial representations of fixed degree from their dominant highest weights (when restricted to $SL_2(\mathbb{F})$) $\mu, \lambda \in \ZZ_{\geq 0}$. This algorithm is well-known, see e.g. \cite[Section 1]{AP}.

\subsection{Pseudocode}
\begin{algorithmic}[1]
	\renewcommand{\algorithmiccomment}[1]{\qquad{\color{grey}// #1}}
	\newcommand{\FUNCTION}[1]{\STATE {\textbf{function}} #1 \begin{ALC@g}}
	\newcommand{\ENDFUNCTION}{\end{ALC@g} \STATE \textbf{end function}}
	\STATE {\textbf{function}} lambdamu($p$, $\lambda$, $\mu$) 
	\begin{ALC@g}
		\STATE $a \leftarrow \lfloor\lambda/p\rfloor$
		\STATE $b \leftarrow \lfloor\mu/p\rfloor$
		\STATE $i \leftarrow \lambda \textbf{ mod } p$
		\STATE $j \leftarrow \mu \textbf{ mod } p$
		\IF {($i=p-1$ and $j\le p-2$) or ($i\le p-2$ and $j=p-1$)}
		\RETURN "the dimension equals 0 for all $k$"
		\ELSIF {$i=p-1$ and $j=p-1$}
		\RETURN lambdamu($p$, $a$, $b$)
		\ELSIF {($a\equiv b \bmod{2}$ and $i=j$) or ($a\not\equiv b\bmod{2}$ and $i=p-2-j$)}
			\STATE $e \leftarrow a+1$
			\STATE $m \leftarrow b+1$
			\RETURN $e,m$
		\ELSE
		\RETURN "the dimension equals 0 for all $k$"
		\ENDIF
	\ENDFUNCTION
\end{algorithmic}
 
\subsection{Sourcecode}
$ $\\
\begin{lstlisting}[title={lambdamu.c}]
#include <stdlib.h>
#include <stdio.h>
#include <string.h>
#include <math.h>
#include <stdint.h>


int lambdamu(intmax_t * m, intmax_t * e, intmax_t p, intmax_t l, intmax_t u) {
	intmax_t i, j, a, b;
	
	a = l;
	b = u;
	
	do {
		l = a;
		u = b;
		
		i = l % p;
		j = u % p;
	
		a = (l - i) / p;
		b = (u - j) / p;
	
		if ((i == (p-1) && j <= (p-2)) || (i <= (p-2) && j == (p-1))) {
			printf("The dimension equals 0 for all k\n");
			return -1;
		}
	} while (i == (p-1) && j == (p-1));
	
	if ((((a-b) % 2) == 0 && i == j) || ((a-b) % 2 && i == (p-2-j))) {
		*e = a+1;
		*m = b+1;
		return 0;
	} else {
		printf("The dimension equals 0 for all k\n");
		return -1;
	}
	
}


int main(int argc, char ** argv) {
	
	intmax_t p = -1;
	intmax_t l = -1;
	intmax_t u = -1;
	
	fprintf(stderr, "Warning: p has to be odd prime >= 2. This is not checked\n");
	
	if (argc > 1) {
		argv++;
		while (*argv) {
			switch (**argv) {
				case '-':
				switch ((*argv)[1]) {
					case 'p': p = atoi(*(++argv)); break;
					case 'l': l = atoi(*(++argv)); break;
					case 'u': u = atoi(*(++argv)); break;
					default:
					fprintf(stderr, "Unsupported argument %s\n", *argv);
					exit(-1);
				}
				break;
				default:
				fprintf(stderr, "Unsupported argument %s\n", *argv);
				exit(-1);
			}
			argv++;
		}
		
	}
	
	if (p < 2) {
		printf("input p: ");
		fflush(stdout);
		scanf("%ld", &p);
	}
	
	while (l <= 1) {
		printf("input lambda: ");
		fflush(stdout);
		scanf("%ld", &l);
	}
	
	while (u < 0) {
		printf("input mu: ");
		fflush(stdout);
		scanf("%ld", &u);
	}
	
	intmax_t m, e;
	
	if (!lambdamu(&m, &e, p, l, u)) {
		printf("m = %ld\ne = %ld\n", m, e);
	}
	
	return 0;
	
}




\end{lstlisting}

\end{document}